\documentclass[11pt,a4paper,reqno]{amsart}
\usepackage[margin=2cm]{geometry}
\usepackage[shortlabels]{enumitem}
\usepackage{microtype,
amsrefs,
mathtools,
xcolor
}

\usepackage{
amsmath,
amsfonts,
amssymb,
amscd,
amsthm, 
calc
}

\newtheorem{theorem}{Theorem}[section]
\newtheorem{proposition}[theorem]{Proposition}
\newtheorem{lemma}[theorem]{Lemma}

\newtheorem{claim}{Claim}
\newtheorem*{claim*}{Claim}

\theoremstyle{definition}
\newtheorem{definition}[theorem]{Definition}

\newcommand{\restr}{\mbox{\raisebox{.5mm}{$\upharpoonright$}}}
\newcommand{\la}{\langle}
\newcommand{\ra}{\rangle}

\newcommand{\bigset}[1]{\big\{ #1 \big\}}
\newcommand{\ex}{\exists}
\newcommand{\fa}{\forall}
\newcommand{\aaa}{\wedge}

\renewcommand{\leq}{\leqslant}
\renewcommand{\geq}{\geqslant}

\newcommand{\A}{\mathcal{A}}
\newcommand{\B}{\mathcal{B}}
\newcommand{\K}{\mathcal{K}}
\newcommand{\Keff}{\mathcal{K}_2^\mathrm{eff}}
\renewcommand{\O}{\mathcal{O}}

\newcommand{\Piz}{\Pi^0_1}
\newcommand{\OCK}{{\omega_1^\textit{CK}}}
\newcommand{\ord}{\mathrm{ord}}
\newcommand{\PP}{\mbox{$\mathcal P$}}

\DeclareMathOperator{\ran}{\mathrm{ran}}

\newcommand{\andd}{\wedge}
\newcommand{\orr}{\vee}
\newcommand{\imp}{\rightarrow}
\newcommand{\Imp}{\Rightarrow}

\newcommand{\Biimp}{\Leftrightarrow}

\newcommand{\smf}{\smallfrown}
\newcommand{\napprox}{\not\approx}
\newcommand{\nsimeq}{\not\simeq}

\newcommand{\da}{{\downarrow}}
\newcommand{\ua}{{\uparrow}}

\newcommand{\ltKB}{<_\mathrm{KB}}

\newcommand{\leqm}{\leq_\mathrm{m}}
\newcommand{\equivm}{\equiv_\mathrm{m}}

\newcommand{\leqT}{\leq_\mathrm{T}}

\AtBeginDocument{%
   \def\MR#1{}
}

\begin{document}

\title[Ordinal analysis of pca's]
{Ordinal analysis of partial combinatory algebras}

\author[P. Shafer]{Paul Shafer}
\address[Paul Shafer]{School of Mathematics\\
University of Leeds\\
Leeds\\
LS2 9JT\\
United Kingdom}
\email{p.e.shafer@leeds.ac.uk}
\urladdr{http://www1.maths.leeds.ac.uk/~matpsh/}

\author[S. A. Terwijn]{Sebastiaan A. Terwijn}
\address[Sebastiaan A. Terwijn]{Radboud University Nijmegen\\
Department of Mathematics\\
P.O. Box 9010, 6500 GL Nijmegen, the Netherlands.
} \email{terwijn@math.ru.nl}

\begin{abstract} 
For every partial combinatory algebra (pca), we define a 
hierarchy of extensionality relations using ordinals. 
We investigate the closure ordinals of pca's, i.e.\ the 
smallest ordinals where these relations become equal. 
We show that the closure ordinal of Kleene's first model 
is $\OCK$ and that the closure ordinal of Kleene's second model is~$\omega_1$.
We calculate the exact complexities of the extensionality relations in Kleene's first model, showing that they exhaust the hyperarithmetical hierarchy.  We also discuss embeddings of pca's. 
\end{abstract}



\date{\today}

\maketitle

\section{Introduction}

Partial combinatory algebras (pca's) were introduced by Feferman~\cite{Feferman}
in connection with the study of predicative systems of mathematics. 
Since then, they have been studied as abstract models of computation, 
in the same spirit as combinatory algebras (defined and studied long before, 
in the 1920's, by Sch\"onfinkel and Curry) and the closely related
lambda calculus, cf.\ Barendregt~\cite{Barendregt}. 
As such, pca's figure prominently in the literature on constructive mathematics, 
see e.g.\ Beeson~\cite{Beeson} and 
Troelstra and van Dalen \cite{TroelstravanDalenII}.
This holds in particular for the theory of realizability, see
van Oosten \cite{vanOosten}. 
For example, pca's serve as the basis of various models of 
constructive set theory, first defined by McCarty~\cite{McCarty}.
See Rathjen~\cite{Rathjen} for further developments using this 
construction. 

A pca is a set $\A$ equipped with a partial application operator $\cdot$
that has the same properties as a classical (total) combinatory 
algebra (see Section~\ref{sec:prelim} below for precise definitions). 
In particular, it has the combinators $K$ and $S$. 
Extensionality is a key notion in the lambda calculus and the 
study of its models. 
A pca is called \emph{extensional} if for any of its elements $f$ and $g$,
$f=g$ whenever $fx \simeq gx$ for every~$x$.  
Here $\simeq$ denotes Kleene equality:  either both sides are undefined 
or both sides are defined and equal.
In \cite{BarendregtTerwijn2019} this property was studied in connection 
with generalized numberings. 
Below we define a hierarchy of extensionality relations using ordinals, 
and we determine for various pca's the ordinal for which this hierarchy
becomes stable. In particular we do this for Kleene's 
first model $\K_1$ (which is the setting of ordinary computability
theory) and Kleene's second model~$\K_2$.

The paper is organized as follows. 
In Section~\ref{sec:prelim} we present the necessary preliminaries on pca's. 
In Section~\ref{sec:higher} we define the higher extensionality relations $\sim_\alpha$ and their limit $\approx$, and we prove some of their general properties.
In Section~\ref{sec:ordinal} we present preliminaries on 
constructive ordinals and hyperarithmetical sets. 
In Section~\ref{sec:K1} we show that the closure ordinal 
of Kleene's first model $\K_1$ is $\OCK$, the first nonconstructive 
ordinal.  Thus ${\approx} = {\sim_\OCK}$ in $\K_1$.  We also show that the relation $\approx$ on $\K_1$ is $\Pi^1_1$-complete.
In Section~\ref{sec:arithmetical} and Section~\ref{sec:hyp} we calculate the complexities of the relations $\sim_\alpha$ on $\K_1$ for $0 < \alpha < \OCK$.  Specifically, we show that for $0 < \alpha < \OCK$:
\begin{itemize}
\item the relation $\sim_\alpha$ is $\Pi^0_{1 + F(\alpha)}$-complete if $\alpha$ is a successor ordinal, and
\item the relation $\sim_\alpha$ is $\Sigma^0_{1 + F(\alpha)}$-complete if $\alpha$ is a limit ordinal,
\end{itemize}
where $F$ is the function given by
\begin{align*}
F(0) &= 0 & \\ 
F(\alpha + d + 1) &= F(\alpha) + 1 & &\text{if $d < \omega$ and either $\alpha = 0$ or $\alpha$ is a limit ordinal}\\
F(\alpha) &= \lim_{\beta < \alpha}(F(\beta) + 1)& &\text{if $\alpha$ is a limit ordinal}.
\end{align*}
The proof is not uniform.  Section~\ref{sec:arithmetical} handles ordinals $0 < \alpha < \omega^2$, which correspond to arithmetical complexities, while Section~\ref{sec:hyp} handles ordinals $\omega^2 \leq \alpha < \OCK$, which correspond to non-arithmetical complexities.
In Section~\ref{sec:embeddings} we study embeddings of pca's. 
In Section~\ref{sec:K2} we show that the closure ordinal 
for Kleene's second model $\K_2$ is $\omega_1$
and that the closure ordinal of its effective version $\Keff$ is again~$\OCK$.

Our notation from computability theory is mostly standard.
In the following, $\omega$ denotes the natural numbers, and 
$\Phi_e$ denotes the $e$\textsuperscript{th} partial computable (p.c.) function. 
For unexplained notions from computability theory we refer to 
Odifreddi~\cite{Odifreddi} and Soare~\cite{SoareBookRE}.
For background on the lambda calculus we refer to Barendregt~\cite{Barendregt}.

\section{Preliminaries on pca's} \label{sec:prelim}

A \emph{partial applicative structure} is a set $\A$ 
with a partial application operator $\cdot$, which is 
a partial map from $\A\times\A$ to $\A$. 
We mostly write $ab$ instead of $a \cdot b$, though we do sometimes include the $\cdot$ in places where it improves readability. 
When $ab$ is defined, i.e.\ when the pair $(a,b)$ is in the domain of the application operator, we write $ab\da$.  If $ab$ is not defined, we write $ab\ua$.
An element $f\in\A$ is \emph{total} if 
$fa\da$ for every $a\in\A$, and $\A$ is \emph{total} if all of its elements are total.
Application associates to the left, and we write 
$abc$ instead of $(ab)c$.

The set of terms over $\A$ is defined inductively as follows:
\begin{itemize}
\item every element $a\in\A$ is a term,
\item every variable $v$ is a term, and
\item if $s$ and $t$ are terms, then $(s \cdot t)$ is a term. 
\end{itemize}
As usual, a term $t$ is \emph{closed} if it does not contain any 
variables.
We use Kleene equality for closed terms:
$s \simeq t$ means that either both $s$ and $t$ are undefined, 
or both are defined (meaning all their subterms are defined)
and evaluate to the same element of $\A$. 

The property that makes a partial applicative structure 
a partial combinatory algebra is \emph{combinatory completeness}, 
which loosely means that every multivariate function defined by a term 
is represented by an element of the algebra. 
Feferman~\cite{Feferman} proved that this is equivalent to the 
existence of the classical combinators $K$ and $S$, so we can 
also take this as a definition.

\begin{definition}\label{def-pca}
A partial applicative structure $\A$ is a \emph{partial combinatory algebra} (\emph{pca}) if it has elements $K$ and $S$
with the following properties for all $a,b,c\in\A$:
\begin{itemize}

\item $Kab\da = a$, 

\item $Sab\da$ and $Sabc \simeq ac(bc)$.

\end{itemize}
\end{definition}

The most important example of a pca is \emph{Kleene's first model} $\K_1$, 
consisting of $\omega$ with application defined as $n\cdot m = \Phi_n(m)$.
The combinators $K$ and $S$ can easily be defined using the 
S-m-n-theorem.
For any $X \subseteq \omega$, relativizing the partial computable functions to oracle $X$ yields a pca $\K_1^X$ with application given by $n \cdot m = \Phi_n^X(m)$.

\emph{Kleene's second model} $\K_2$ \cite{KleeneVesley}
consists of Baire space $\omega^\omega$, 
with application $f\cdot g$ defined by applying the continuous 
functional with code $f$ to the input $g$. 
For details of the coding see Longley and Normann~\cite{LongleyNormann}*{Section~3.2.1}.
It is a bit more work to define the combinators $K$ and $S$ in this case. 
For the material below, the precise details of the coding are 
largely irrelevant. 
An alternative coding of $\K_2$ can be given as follows.
Viewing $\Phi_e^f$ as a function of the oracle $f \in \omega^\omega$, we define application as 
\begin{equation} \label{alt}
f\cdot g = \Phi^{f\oplus g}_{f(0)},
\end{equation}
where it is understood that $f\cdot g$ is defined 
if and only if the function computed on the right is total.  
This coding is different from the standard definition of $\K_2$, 
but equivalent to it in the following sense. 
In both codings, every partial continuous functional on $\omega^\omega$ 
is represented, and the same holds for multivariate functions 
on $\omega^\omega$. Furthermore, effective translations from 
one coding to the other can be given in both directions. 
(The latter fact does not seem to imply that \eqref{alt}
is also a pca in a straightforward way, but this can be 
proved directly in much the same way as for the standard 
coding, by providing definitions of the combinators 
$K$ and $S$.)

\section{Higher extensionality} \label{sec:higher}

The notion of extensionality plays an important part in the lambda calculus, 
cf.\ \cite{Barendregt}. 
In \cite{BarendregtTerwijn} various notions of extensionality were studied 
in connection with the precompleteness of numberings based on pca's. 
Here we define a hierarchy of extensionality relations using ordinals as follows. 

\begin{definition} \label{hierarchy}
Given a pca $\A$, we define equivalence relations $\sim_\alpha$ on the closed terms over $\A$ for all ordinals $\alpha$.  For all closed terms $s$ and $t$ over $\A$, define:
\begin{align*}
s \sim_0 t &\Longleftrightarrow s \simeq t  \\
s \sim_1 t &\Longleftrightarrow \fa x\in\A \; s x \simeq t x \\
s \sim_{\alpha+1} t &\Longleftrightarrow \fa x\in\A \; s x \sim_\alpha t x \\
s \sim_\alpha t &\Longleftrightarrow \ex \beta<\alpha \; s \sim_\beta t &&\text{for $\alpha$ a limit ordinal} 
\end{align*}
\end{definition}
We sometimes restrict $\sim_\alpha$ to $\A$ and consider it as a relation on the elements of $\A$ rather than on the closed terms over $\A$.  Notice that $\A$ is extensional if and only if the relations $\sim_0$ and $\sim_1$ are equal when restricted to $\A$.  
To see that the relations $\sim_\alpha$ are indeed equivalence relations, see 
the discussion following Theorem~\ref{thm:basic}.

Under Definition~\ref{hierarchy} it is possible that an undefined term 
is equivalent to a defined one. 
For example, if $t\ua$ and $s$ is such that $s\da$ but 
$\fa x \; sx\ua$, then $s \sim_1 t$.

For any ordinal $\alpha$, a straightforward induction on $n \in \omega$ shows that, for any closed terms $s$ and $t$,
\begin{align*}
s \sim_{\alpha + n} t &\Longleftrightarrow \fa x_0, \dots, x_{n-1} (s x_0 \cdots x_{n-1} \sim_\alpha t x_0 \cdots x_{n-1}).
\intertext{In particular,}
s \sim_n t &\Longleftrightarrow \fa x_0, \dots, x_{n-1} (s x_0 \cdots x_{n-1} \simeq t x_0 \cdots x_{n-1})
\end{align*}
for every $n$.

Also observe that $s \sim_n t \Rightarrow s \sim_{n+1} t$ for every~$n$.  Theorem~\ref{thm:omega} below shows that the converse fails for non-extensional pca's, such as $\K_1$ and $\K_2$.  More generally, for every pca we have that $s \sim_\beta t \Rightarrow s \sim_\alpha t$ whenever $\beta \leq \alpha$.  Thus the $\sim_\alpha$ form an ascending sequence of relations.

\begin{theorem} \label{thm:basic}
The following hold for every pca $\A$, every pair of closed terms $s$ and $t$ over $\A$, and every ordinal $\alpha$.
\begin{enumerate}[(i)]

\item\label{it-UpOne} $s \sim_\alpha t \Rightarrow s x \sim_\alpha t x$ for all~$x\in\A$. 

\item\label{it-PresUp} $s \sim_\beta t \Rightarrow s \sim_\alpha t$ for all ordinals 
$\alpha\geq\beta$.

\end{enumerate}
\end{theorem}
\begin{proof}
Note that we can rephrase~\ref{it-UpOne} as 
$s \sim_\alpha t \Rightarrow s \sim_{\alpha+1} t$.
We prove this by transfinite induction on $\alpha$. 
For $\alpha=0$ this is clear because $s \simeq t$ implies $s x \simeq t x$ 
for every~$x$.
For a successor $\alpha+1$, suppose that $s \sim_{\alpha+1} t$, i.e.\ 
$\fa x \; s x \sim_\alpha t x$. By the induction hypothesis, we then have 
$\fa x \fa y \; s x y \sim_\alpha t x y$, hence $s x \sim_{\alpha+1} t x$ for every~$x$.
Finally, suppose that $s \sim_\alpha t$ for $\alpha$ a limit. 
Then $s \sim_\beta t$ for some $\beta<\alpha$, so by the induction hypothesis
$s x \sim_\beta t x$ for every $x$, hence $s x \sim_\alpha t x$ for every $x$ by the definition 
of $\sim_\alpha$.

For~\ref{it-PresUp}, we again proceed by transfinite induction on $\alpha$. 
For $\alpha=0$ this is trivial. 
The case where $\alpha$ is a successor follows from~\ref{it-UpOne} and the induction hypothesis. 
Finally, if $\alpha$ is a limit and $s \sim_\beta t$ with $\beta \leq \alpha$, 
then $s \sim_\alpha t$ by the definition of $\sim_\alpha$.
\end{proof}

Given a pca $\A$, we can use Theorem~\ref{thm:basic} and transfinite induction on $\alpha$ to verify that each $\sim_\alpha$ is an equivalence relation.  It is easy to see that $\sim_0$ is an equivalence relation.  That $\sim_{\alpha + 1}$ is an equivalence relation follows straightforwardly from the induction hypothesis.  When $\alpha$ is a limit, that $\sim_\alpha$ is reflexive and symmetric also follows straightforwardly from the induction hypothesis.  For transitivity, suppose that $r \sim_\alpha s \sim_\alpha t$.  Then there are $\beta_0, \beta_1 < \alpha$ such that $r \sim_{\beta_0} s$ and $s \sim_{\beta_1} t$.  Let $\beta = \max\{\beta_0, \beta_1\}$.  We then have that $\beta < \alpha$, that $r \sim_\beta s \sim_\beta t$ by Theorem~\ref{thm:basic}, and thus that $r \sim_\beta t$ by the induction hypothesis.  So $r \sim_\alpha t$.  

\begin{lemma}\label{lem-AddK}
Let $\A$ be a pca, and let $s$ and $t$ be closed terms over $\A$.  Then for every ordinal $\alpha$, $s \sim_\alpha t$ if and only if $K s \sim_{\alpha + 1} K t$.
\end{lemma}

\begin{proof}
Observe that if $t$ is a closed term and $t \ua$, then $K t \ua$ as well, and therefore also $K t x \ua$ for every $x \in \A$.  Thus $K t x \simeq t$ for every closed term $t$ and every $x \in \A$, even when $t \ua$.

Let $s$ and $t$ be closed terms.  First suppose that $s \sim_\alpha t$.  Then
\begin{align*}
K s x \sim_0 s \sim_\alpha t \sim_0 K t x
\end{align*}
for every $x$, so $K s \sim_{\alpha + 1} K t$.

Conversely, suppose that $K s \sim_{\alpha + 1} K t$.  Choose any $x \in \A$.  Then 
\begin{align*}
s \sim_0 K s x \sim_\alpha K t x \sim_0 t,
\end{align*}
so $s \sim_\alpha t$.
\end{proof}

\begin{theorem} \label{thm:omega}
Suppose that $\A$ is not extensional. Then
the relations $\sim_\alpha$ for $\alpha \leq \omega$   
are all different, even when restricted to $\A$.
\end{theorem}
\begin{proof}
Since $\A$ is not extensional,  
the relations $\sim_0$ and $\sim_1$ differ on $\A$.
Suppose that $d,e\in \A$ are such that $d\neq e$ but $d\sim_1 e$.

Inductively define elements $f_n$ and $g_n$ of $\A$ for $n \in \omega$ as follows,
using the combinator~$K$.
\begin{align*}
f_0 &= d & g_0 &= e \\
f_{n+1} &= K f_n & g_{n+1} &= K g_n
\end{align*}

We have that $f_0 \not\sim_0 g_0$ and $f_0 \sim_1 g_0$ by the choice of 
$f_0$ and $g_0$.  By induction on $n$ and Lemma~\ref{lem-AddK}, it follows that 
$f_n \not\sim_n g_n$ and $f_n \sim_{n+1} g_n$ for every $n$.  Thus $\sim_{n+1}$ is a strictly weaker relation than $\sim_n$ for every $n$.

It also follows that $\sim_\omega$ strictly weaker than $\sim_n$ for every~$n$.
Suppose that there is a fixed $n$ such that $f\sim_\omega g \Rightarrow f\sim_n g$ for every $f$ and $g$.
Then
$$
f\sim_{n+1} g \Rightarrow f\sim_\omega g \Rightarrow f\sim_n g 
$$
for every $f$ and $g$, 
contradicting that $\sim_{n+1}$ is strictly weaker than $\sim_n$.
\end{proof}

However, for cardinality reasons, the relations $\sim_\alpha$ cannot all be different on a given pca.

\begin{theorem} \label{thm:fix}
For every pca $\A$, there is an ordinal $\alpha$ such that $\sim_\alpha$ is equal to $\sim_{\alpha + 1}$.  Furthermore, the least such $\alpha$ is either $0$ or a limit ordinal.
\end{theorem}

\begin{proof}
Let $\Omega$ denote the set of closed terms over $\A$, and view each relation $\sim_\alpha$ as a subset of $\Omega \times \Omega$.  Theorem~\ref{thm:basic} tells us that these relations form an ascending sequence:  $\alpha \leq \gamma \Rightarrow {\sim_\alpha} \subseteq {\sim_\gamma}$ for all ordinals $\alpha$ and $\gamma$.  Therefore there must be an $\alpha < |\Omega \times \Omega|^+$ for which ${\sim_\alpha} = {\sim_{\alpha+1}}$ because ${\sim_\alpha} \neq {\sim_{\alpha+1}}$ implies that there is an $(s, t) \in \Omega \times \Omega$ such that $s \sim_{\alpha+1} t$ but $(\forall \beta \leq \alpha)(s \not\sim_\beta t)$, and there are only $|\Omega \times \Omega|$ many elements of $\Omega \times \Omega$ to choose among.

Now suppose that ${\sim_{\beta+1}} = {\sim_{\beta+2}}$ for some ordinal $\beta$.  We show that ${\sim_\beta} = {\sim_{\beta+1}}$ as well.  It follows that the least $\alpha$ such that ${\sim_\alpha} = {\sim_{\alpha+1}}$ cannot be a successor.  Consider closed terms $s$ and $t$.  We already know that $s \sim_\beta t \Imp s \sim_{\beta+1} t$.  So suppose that $s \sim_{\beta+1} t$.  Then $K s \sim_{\beta+2} K t$ by Lemma~\ref{lem-AddK}, so $K s \sim_{\beta+1} K t$ by the assumption that ${\sim_{\beta+1}} = {\sim_{\beta+2}}$.  Thus $s \sim_\beta t$, again by Lemma~\ref{lem-AddK}.  Therefore $s \sim_{\beta+1} t \Imp s \sim_\beta t$, so ${\sim_\beta} = {\sim_{\beta+1}}$.

Thus there is an $\alpha$ such that ${\sim_\alpha} = {\sim_{\alpha+1}}$, and the least such $\alpha$ is either $0$ or a limit ordinal.
\end{proof}

\begin{definition}\label{def-OrdA}
For a pca $\A$, define $\ord(\A)$ to be the least ordinal $\alpha$ such that ${\sim_\alpha} = {\sim_{\alpha + 1}}$.
\end{definition}

Let $\A$ be a pca.  Observe that if ${\sim_\alpha} = {\sim_{\alpha+1}}$, then an easy transfinite induction on $\gamma \geq \alpha$ shows that ${\sim_\gamma} = {\sim_\alpha}$ for all $\gamma \geq \alpha$.  Therefore $\ord(\A)$ is also the least ordinal $\alpha$ such that $(\forall \gamma \geq \alpha)({\sim_\gamma} = {\sim_\alpha})$.

For closed terms $s$ and $t$, write $s \approx t$ if there is a $\gamma$ such that $s \sim_\gamma t$.  Let $\alpha = \ord(\A)$.  Then $s \approx t \Biimp s \sim_\alpha t$.  That $s \sim_\alpha t \Imp s \approx t$ is clear from the definition of $\approx$.  Conversely, suppose that $s \approx t$.  
Then there is a least $\beta$ such that $s \sim_\beta t$.  
As $(\forall \gamma \geq \alpha)({\sim_\gamma} = {\sim_\alpha})$, 
we must have $\beta \leq \alpha$, and hence $s \sim_\alpha t$ by Theorem~\ref{thm:basic}.  
Thus $s \approx t \Imp s \sim_\alpha t$, so $s \approx t \Biimp s \sim_\alpha t$.  Therefore ${\approx}$ and ${\sim_\alpha}$ are the same relation.  Furthermore, if $\alpha > 0$, then it is a limit ordinal, and so for any particular closed terms $s$ and $t$, we have that $s \approx t$ if and only if there is a $\beta < \alpha$ such that $s \sim_\beta t$.

If $\A$ is a pca and $t$ is a closed term, then there is an $f \in \A$ with $f \sim_1 t$.  This follows from the combinatory completeness of $\A$ because $\A$ must contain an element $f$ representing the term $t v$, where $v$ is a variable.  Alternatively, one may argue as follows.  If $t \da$, then $t$ evaluates to some $f \in \A$.  Thus $f \sim_0 t$, so $f \sim_1 t$.  If $t \ua$, then $\A$ must have non-total elements.  In fact, there must be an $f \in \A$ such that $f x \ua$ for all $x \in \A$, in which case $f \sim_1 t$.  Therefore, if $\gamma > \alpha \geq 1$ and ${\sim_\alpha} \neq {\sim_\gamma}$, then this inequality is witnessed by members of $\A$.  If ${\sim_\alpha} \neq {\sim_\gamma}$, then there are closed terms $s$ and $t$ such that 
$s \not\sim_\alpha t$ but $s \sim_\gamma t$.  Let $f$ and $g$ be elements of $\A$ where $f \sim_1 s$ and $g \sim_1 t$.  As $\alpha \geq 1$, we also have that $f \sim_\alpha s$ and $g \sim_\alpha t$, 
so $f \not\sim_\alpha g$.
Likewise, $f \sim_\gamma s \sim_\gamma t \sim_\gamma g$, so $f \sim_\gamma g$.  So $f$ and $g$ are elements of $\A$ such that $f \not\sim_\alpha g$ but $f \sim_\gamma g$.  Additionally, if $\A$ is not extensional, then by definition there are $f, g \in \A$ such that $f \not\sim_0 g$ but $f \sim_1 g$.  Thus if $\A$ is not extensional and ${\sim_\alpha} \neq {\sim_\gamma}$ for some $\gamma > \alpha \geq 0$, then the inequality is witnessed by members of $\A$.

We therefore have the following equivalent characterizations of $\alpha = \ord(\A)$.
\begin{itemize}
\item $\alpha$ is least such that ${\sim_\alpha} = {\sim_{\alpha + 1}}$.

\item $\alpha$ is least such that ${\sim_\gamma} = {\sim_\alpha}$ for all $\gamma \geq \alpha$.

\item $\alpha$ is least such that ${\sim_\alpha} = {\approx}$.
\end{itemize}
Furthermore, if $\A$ is not extensional, then the above items also characterize $\ord(\A)$ when the relations are all restricted to $\A$.  In particular, if $\A$ is not extensional, then it makes no difference whether we define $\ord(A)$ by considering the $\sim_\alpha$'s as relations on closed terms or as relations on elements of $\A$.  If $\A$ is extensional, then the situation is slightly more nuanced.  If $\A$ is extensional and total, then every closed term is $\sim_0$-equivalent to some member of $\A$, so ${\sim_0} = {\sim_1}$ as relations on closed terms, and therefore $\ord(\A) = 0$.  However, if $\A$ is extensional but has non-total elements, then there is a (unique) $f \in \A$ such that $f x \ua$ for all $x \in \A$.  In this case we have, for example, that $K f \not\sim_1 f$ but $K f \sim_2 f$.  Therefore ${\sim_1} \neq {\sim_2}$, even when restricted to $\A$. 
Thus $\ord(\A) \geq \omega$ by Theorem~\ref{thm:fix} because $\ord(A)$ must be a limit ordinal.  None of the pca's we consider are extensional.  However, there are interesting examples of extensional pca's with non-total elements~\cite{Bethke}.

\section{Constructive ordinals} \label{sec:ordinal}

In this section we collect some material on constructive ordinals 
and hyperarithmetical sets that we will use in the following. 
We use the notation from Sacks~\cite{Sacks}, to which we also 
refer for more details. 
A thorough historical account of hyperarithmetical sets can be 
found in Moschovakis~\cite{Moschovakis}.

Kleene~\cite{Kleene1938} introduced a notation system $\O\subseteq\omega$ 
for constructive ordinals, starting with a notation for the ordinal $0$ and closing under
successor and effective limits.\footnote{For a precise definition 
and intricacies of the coding see Moschovakis~\cite{Moschovakis}.}
For every ordinal notation $x\in\O$, 
$|x|$ denotes the corresponding constructive ordinal.
Given an $x \in \O$, the notation system allows us to effectively determine whether $|x|$ is a successor ordinal or a limit ordinal, whether $|x|$ is an even ordinal or an odd ordinal, a notation for $|x|$'s successor, and, if $|x|$ is a successor, a notation for $|x|$'s predecessor.
The constructive ordinals form an initial segment of the set of countable ordinals, 
and their supremum is called $\OCK$.
The set $\O$ is equipped with a partial order $<_\O$ such that 
$a <_\O b$ implies $|a| < |b|$.
The main tool here is effective transfinite recursion,
which allows one to define computable objects by recursion along $<_\O$, 
even though the relation $a <_\O b$ itself is not computable. 

The class of constructive ordinals can also be characterized using 
computable well-orders. An ordinal is called computable if it is finite 
or the order type of a computable well-order on $\omega$. 
Markwald and Spector independently proved that the 
computable ordinals equal the constructive ordinals 
(cf.\ \cite{Moschovakis}*{Theorem 2A.1}).
We often blur the distinctions between a constructive ordinal and a computable ordinal and between a constructive ordinal and its notation.

Using recursion on constructive ordinals, the arithmetical 
hierarchy can be extended into the transfinite, yielding 
the hyperarithmetical hierarchy.
The class of hyperarithmetical sets is the smallest class 
containing the computable sets that is closed under 
complementation and effective unions.\footnote{
Moschovakis~\cite{Moschovakis} points out that this is 
probably due to Shoenfield.}
Finally, Kleene proved that a set is hyperarithmetical if and only if it 
is $\Delta^1_1$. (Since this is an effective version of the classical 
result that the $\mathbf{\Delta}^1_1$ sets are exactly the Borel sets, this is 
sometimes called the Kleene-Suslin theorem.)

The hyperarithmetical hierarchy is stratified by the computable infinitary formulas.  Essentially, a $\Sigma^0_0$ / $\Pi^0_0$ formula is coded as an index for a machine computing the characteristic function of a relation; a $\Sigma^0_\alpha$ formula is coded as a c.e.\ disjunction of $\Pi^0_\beta$ formulas for $\beta < \alpha$; and a $\Pi^0_\alpha$ formula is coded as a c.e.\ conjunction of $\Sigma^0_\beta$ formulas for $\beta < \alpha$.  The following definition is presented as in~\cite{HirschfeldtWhite}.  See~\cite{AshKnight}*{Chapter~7} for a more detailed treatment of computable infinitary formulas.

\begin{definition}\
\begin{itemize}
\item A $\Sigma^0_0$ ($\Pi^0_0$) index for a computable formula $\varphi(\bar n)$ is a triple $\la \Sigma, 0, e \ra$ ($\la \Pi, 0, e \ra$), where $e$ is the index of a total computable function $\Phi_e(\bar n)$ computing the characteristic function of the relation defined by $\varphi$.

\item For a computable ordinal $\alpha$, a $\Sigma^0_\alpha$ index for a computable formula $\varphi(\bar n)$ is a triple $\la \Sigma, a, e \ra$, where $a$ is a notation for $\alpha$ and $e$ is an index for a c.e.\ set of $\Pi^0_{\beta_k}$ indices for computable formulas $\psi_k(\bar n, \bar x)$, where $\beta_k < \alpha$ for each $k$ and
\begin{align*}
\varphi(\bar n) \equiv \bigvee_{k \in \omega} \exists \bar{x} \; \psi_k(\bar n, \bar x).
\end{align*}

\item For a computable ordinal $\alpha$, a $\Pi^0_\alpha$ index for a computable formula $\varphi(\bar n)$ is a triple $\la \Pi, a, e \ra$, where $a$ is a notation for $\alpha$ and $e$ is an index for a c.e.\ set of $\Sigma^0_{\beta_k}$ indices for computable formulas $\psi_k(\bar n, \bar x)$, where $\beta_k < \alpha$ for each $k$ and
\begin{align*}
\varphi(\bar n) \equiv \bigwedge_{k \in \omega} \forall \bar{x} \; \psi_k(\bar n, \bar x).
\end{align*}

\item For a computable ordinal $\alpha$, a set $X \subseteq \omega^n$ is said to be $\Sigma^0_\alpha$ ($\Pi^0_\alpha$) if there is a $\Sigma^0_\alpha$ ($\Pi^0_\alpha$) formula defining $X$.
\end{itemize}
\end{definition}

At the arithmetical levels, it makes no difference whether one defines the $\Sigma^0_{n+1}$ / $\Pi^0_{n+1}$ sets in terms of ordinary finitary formulas or computable infinitary formulas.  For each $n \geq 1$, every infinitary $\Sigma^0_n$ ($\Pi^0_n$) formula is equivalent to a finitary $\Sigma^0_n$ ($\Pi^0_n$) formula.  For $n=0$, the equivalence depends on one's precise definition of the finitary $\Sigma^0_0$ / $\Pi^0_0$ formulas.

Spector~\cite{Spector}
proved that every hyperarithmetical well-order is isomorphic to a 
computable one.
This result does not hold for linear orders in general.  See 
Ash and Knight \cite{AshKnight} for an overview of results by 
Feiner, Lerman, Jockusch and Soare, Downey, and Seetapun, 
culminating in the following result of 
Knight~\cite{AshKnight}*{Theorem 9.15}:
Every nonzero Turing degree contains a linear order. 
However, Spector's result was salvaged for linear orders 
by Montalb\'an~\cite{Montalban}, 
who showed that every hyperarithmetical linear order is 
\emph{equimorphic} to a computable one, where two linear orders 
are equimorphic if they can be embedded into each other.

We will make use of the following result from Spector.
(It occurs somewhat hidden on p162 of \cite{Spector}. 
See also~\cite{Sacks}*{Corollary~I.5.6}.) 

\begin{theorem} \label{bounding}
{\rm ($\Sigma^1_1$-bounding, Spector \cite{Spector})}
Suppose that $X\subseteq \O$ is $\Sigma^1_1$. 
Then there exists $b\in\O$ such that 
$|x| \leq |b|$ for every $x\in X$. 
\end{theorem}

\section{Ordinal analysis of $\K_1$} \label{sec:K1}

In this section we study the higher notions of extensionality from 
Definition~\ref{hierarchy} for Kleene's first model $\K_1$, i.e.\ the standard 
setting of computability theory. 

\begin{theorem} \label{thm:lower}
The relations $\sim_\alpha$ on $\K_1$ for $\alpha \leq \OCK$ are all different.  Thus $\ord(\K_1) \geq \OCK$.
\end{theorem}

\begin{proof}
By effective transfinite recursion, for every $\alpha < \OCK$ we produce $f_\alpha, g_\alpha \in \K_1$ such that $f \not\sim_\alpha g$ and $f \sim_{\alpha + 1} g$.  Therefore ${\sim_\alpha} \neq {\sim_{\alpha + 1}}$ for all $\alpha < \OCK$.

The pca $\K_1$ is not extensional, so for the base case $\alpha = 0$ we may choose $f_0, g_0 \in \K_1$ such that $f_0 \not\sim_0 g_0$ and $f_0 \sim_1 g_0$.

The successor case is similar to the proof of Theorem~\ref{thm:omega}.  Suppose that $\alpha = \beta + 1$ is a successor.  By effective transfinite recursion, compute $f_\beta, g_\beta \in \K_1$ such that $f_\beta \not\sim_\beta g_\beta$ and $f_\beta \sim_\alpha g_\beta$.  Let $f_\alpha = K f_\beta$ and $g_\alpha = K g_\beta$.  Then $f_\alpha \not\sim_\alpha g_\alpha$ and $f_\alpha \sim_{\alpha + 1} g_\alpha$ by Lemma~\ref{lem-AddK}.

Suppose that $\alpha$ is a limit, and let $\beta_0 < \beta_1 < \cdots$ be a computable sequence of ordinals converging to $\alpha$.  By effective transfinite recursion, we can uniformly compute indices $f_{\beta_n}$ and $g_{\beta_n}$ such that 
$f_{\beta_n} \not\sim_{\beta_n} g_{\beta_n}$ and $f_{\beta_n} \sim_{\beta_n + 1} g_{\beta_n}$ for all $n \in \omega$.  Thus we can compute $f_\alpha$ and $g_\alpha$ so that $f_\alpha \cdot n = f_{\beta_n}$ and $g_\alpha \cdot n = g_{\beta_n}$ for all $n$.  To see that $f_\alpha \not\sim_\alpha g_\alpha$, consider a $\beta < \alpha$, and let $n$ be such that $\beta < \beta_n < \alpha$.  Then 
$f_\alpha \cdot n = f_{\beta_n} \not\sim_{\beta_n} g_{\beta_n} = g_\alpha \cdot n$.  Thus there is an $n$ such that $f_\alpha \cdot n \not\sim_{\beta_n} g_\alpha \cdot n$.  Therefore $f_\alpha \not\sim_{\beta_n + 1} g_\alpha$, so also $f_\alpha \not\sim_\beta g_\alpha$.  Thus $f_\alpha \not\sim_\beta g_\alpha$ for every $\beta < \alpha$, so $f_\alpha \not\sim_\alpha g_\alpha$.  On the other hand, for every $n$, $f_\alpha \cdot n = f_{\beta_n} \sim_{\beta_n + 1} g_{\beta_n} = g_\alpha \cdot n$, and therefore $f_\alpha \cdot n \sim_\alpha g_\alpha \cdot n$.  Thus $f_\alpha \sim_{\alpha + 1} g_\alpha$.  This completes the proof.
\end{proof}

By Theorem~\ref{thm:lower}, we have $\ord(\K_1) \geq \OCK$.
Below we prove that $\ord(\K_1) \leq \OCK$, so $\ord(\K_1) = \OCK$.

For some pairs $f$ and $g$ in $\K_1$ we have that
$f \not\sim_\alpha g$ for \emph{all} ordinals $\alpha$, for example
when $f x \ua$ for all $x \in \K_1$ and $g x_0 \cdots x_{n-1}\da$ for all sequences $x_0, x_1, \dots, x_{n-1}$ of elements of $\K_1$.
However, there are also less trivial examples of such pairs 
$f$ and $g$. 

\begin{definition}
We call an element $f$ of a pca $\A$ \emph{hereditarily total} 
if $f x_0 \cdots x_{n-1} \da$ for all $x_0, \dots, x_{n-1} \in \A$.
\end{definition}

Note that in a total pca all elements are hereditarily total.
However, in a non-total pca the combinator $K$ is total 
(since $Ka=a$ for every $a$), but not hereditarily total 
(take $a$ non-total). 

\begin{proposition} 
There is a pair of hereditarily total $f$ and $g$ 
in $\K_1$ such that $f\not\sim_\alpha g$ for all $\alpha$.
\end{proposition}
\begin{proof}
We define $f \neq g$ so that for every $x$, 
$fx\da = f$ and $gx\da = g$. 
This clearly implies that $f\not\sim_\alpha g$ for all $\alpha$.
First define $f$ so that $fx\da = f$ for all $x$ using the 
recursion theorem. Second, we want to define $g\neq f$ so that 
$gx\da = g$ for every $x$. Inspection of the proof of the 
recursion theorem shows that the fixed point can be chosen to be 
different from any given number by using a padding argument. 
This guarantees the existence of $g$. 
\end{proof}

However, if $f, g \in \K_1$ satisfy $f \sim_\alpha g$ for \emph{some} $\alpha$, 
then there is such an $\alpha < \OCK$.

\begin{theorem} \label{thm:upper}
Suppose that $s$ and $t$ are closed terms over $\K_1$ and that $s \sim_\alpha t$ for some ordinal~$\alpha$. 
Then there is an $\alpha < \OCK$ such that $s \sim_\alpha t$.  That is, $\ord(\K_1) \leq \OCK$.
\end{theorem}
\begin{proof}
We prove that for closed terms $s$ and $t$,
\begin{align*}
s \sim_{\OCK+1} t \Longrightarrow s \sim_\OCK t.
\end{align*}

For the purpose of this proof, denote by $\Omega$ the set of 
closed terms over $\K_1$, coded as elements of $\omega$ in some effective way.
Now consider the operator 
$\Gamma \colon \PP(\Omega\times\Omega) \rightarrow \PP(\Omega\times\Omega)$
defined by 
\begin{align*}
\Gamma(X) = \bigset{(s,t) \in\Omega\times\Omega : 
\fa x \in \K_1 \; (sx,tx)\in X}.
\end{align*}
Define $\Gamma_\alpha$ for every ordinal $\alpha$ by 
\begin{align*}
\Gamma_0 &= \bigset{(s,t) : s \simeq t} \\
\Gamma_{\alpha+1} &= \Gamma(\Gamma_\alpha) \\ 
\Gamma_\gamma &= \bigcup_{\alpha < \gamma} \Gamma_\alpha \hspace{1cm} \text{for $\gamma$ a limit ordinal.}
\end{align*}
Note that $s \sim_\alpha t$ if and only if $(s,t) \in \Gamma_\alpha$
and that $\ord(\K_1)$ is the least ordinal $\alpha$ such that 
$\Gamma_{\alpha+1} = \Gamma_\alpha$.
Also note that $\Gamma_\alpha \subseteq \Gamma_{\alpha+1}$ by
Theorem~\ref{thm:basic}.

The complexity of operators such as $\Gamma$ is measured by the 
complexity of the predicate ``$n\in\Gamma(X)$.'' 
Note that $\Gamma$ is a $\Piz$ operator and that 
$\Gamma$ is monotone, i.e.,
$X\subseteq Y$ implies $\Gamma(X)\subseteq\Gamma(Y)$. 
It follows that the closure ordinal of $\Gamma$ is at most $\OCK$ 
because this holds for every $\Piz$ operator on $\PP(\omega)$
(Gandy \cite{Sacks}*{Corollary~III.8.3}), and also for every 
monotone $\Pi^1_1$ operator 
(Spector~\cite{Spector}, cf.\ \cite{Sacks}*{Corollary~III.8.6}).\footnote{As Moschovakis~\cite{Moschovakis} points out, this is not quite explicit in Spector~\cite{Spector}, but he agrees with Sacks's reference to it.}
Typically it is assumed that $\Gamma_0 = \emptyset$, but these closure properties also hold when $\Gamma_0$ is an arithmetical set as it is here.

Since the operator $\Gamma$ above is both $\Piz$ and monotone, 
the proof that $\OCK$ is an upper bound can be somewhat simplified 
as follows.  Suppose that $s \sim_{\OCK+1} t$.  We have to prove that 
$s \sim_\alpha t$ for some computable $\alpha$. 
Note that $s \sim_{\OCK+1} t$ if 
\begin{equation} \label{assumption}
\fa x\in\omega \; \ex \alpha<\OCK \; s x \sim_\alpha t x.
\end{equation} 
This in itself is not enough to conclude that the $\alpha$'s 
have a bound smaller than $\OCK$ because they could form a 
cofinal sequence in $\OCK$.  The question is:  How hard is it to find $\alpha$ for a given~$x$?

Observe that the $\Gamma_\alpha$ are uniformly 
hyperarithmetical for $\alpha <\OCK$.
Namely, by effective transfinite recursion we can define a 
computable function $f$ such that for every $a\in\O$, 
$f(a)$ is a $\Delta^1_1$ index of $\Gamma_{|a|}$
(cf.\ \cite{Sacks}*{Theorem~II.1.5}).
It follows that the statement 
\begin{equation} \label{ass2}
a\in\O \aaa (sx,tx) \in \Gamma_{|a|}
\end{equation}
is a $\Pi^1_1$ predicate of $x$ and $a$. 
Assuming \eqref{assumption}, we have that for every $x$ 
there is an $a$ such that \eqref{ass2} holds. 
Since $\Pi^1_1$ predicates can be uniformized by 
$\Pi^1_1$ predicates (Kreisel \cite{Sacks}*{Theorem~II.2.3}), there is 
a partial $\Pi^1_1$ function $p(x)$ such that 
\begin{equation} \label{ass3}
\fa x \big( p(x) \in\O \aaa (sx,tx) \in \Gamma_{|p(x)|} \big).
\end{equation}
Since $p$ is in fact total, it is also $\Sigma^1_1$ (cf.\ \cite{Sacks}*{Proposition~I.1.7}).
Therefore the set $\{ p(x) : x\in\omega\}$ is a $\Sigma^1_1$ subset of $\O$, 
and it follows from Theorem~\ref{bounding} that there is a computable 
ordinal $\alpha$ such that $|p(x)| < \alpha$ for every~$x$. 
It follows that $\fa x \; (sx,tx) \in \Gamma_\alpha$ and therefore that
$(s,t) \in \Gamma_{\alpha+1}$.  Thus $s \sim_{\alpha+1} t$.
\end{proof}

\begin{theorem} \label{thm:K1}
$\ord(\K_1) = \OCK$.
\end{theorem}

\begin{proof}
This follows from Theorem~\ref{thm:lower} and Theorem~\ref{thm:upper}.
\end{proof}

Thus for $\K_1$, we have that $\approx$ and $\sim_\OCK$ are the same relation, that the relations $\sim_\alpha$ are different for all $\alpha \leq \OCK$, and that the relations $\sim_\alpha$ are equal to $\approx$ for all $\alpha \geq \OCK$.  The proof of Theorem~\ref{thm:upper} can also be used to show that the $\approx$ relation on $\K_1$ is $\Pi^1_1$.  This is because the $\approx$ relation is the same as the $\sim_\OCK$ relation, the $\sim_\OCK$ relation is the $\Gamma_\OCK$ from the proof of Theorem~\ref{thm:upper}, and $\Gamma_\OCK$ is $\Pi^1_1$ again by~\cite{Sacks}*{Corollary~III.8.3}.  We now show that the $\approx$ relation is $\Pi^1_1$-complete.

First we recall some terminology and notation for strings and trees and then introduce a bit of helpful notation.  For strings $\sigma, \tau \in \omega^{<\omega}$, $|\sigma|$ denotes the length of $\sigma$, $\sigma \sqsubseteq \tau$ means that $\sigma$ is an initial segment of $\tau$, $\sigma \sqsubset \tau$ means that $\sigma$ is a proper initial segment of $\tau$, and $\sigma^\smf\tau$ denotes the concatenation of $\sigma$ and $\tau$.  When $\tau = \la x \ra$ has length $1$, we write $\sigma^\smf x$ in place of $\sigma^\smf \la x \ra$.  For $\sigma \in \omega^{<\omega}$ and $n \leq |\sigma|$, $\sigma \restr n$ denotes the initial segment of $\sigma$ of length $n$.  We sometimes think of a function $f \in \omega^\omega$ as an infinite string, in which case $\sigma \sqsubseteq f$ means that the string $\sigma$ is an initial segment of $f$, and $f \restr n$ denotes the initial segment of $f$ of length $n$.  A tree is a set $T \subseteq \omega^{<\omega}$ that is closed under initial segments:  $\forall \sigma, \tau \in \omega^{<\omega} \, ((\tau \in T \andd \sigma \sqsubseteq \tau) \imp \sigma \in T)$.  A function $f \in \omega^\omega$ is a path through the tree $T \subseteq \omega^{<\omega}$ if every initial segment of $f$ is in $T$.

Think of a string $\sigma \in \omega^{<\omega}$ as denoting a sequence of elements of $\K_1$.  Then for a closed term $t$ over $\K_1$ and a $\sigma \in \omega^{<\omega}$, let $t \sigma$ denote
\begin{align*}
t \cdot \sigma(0) \cdot \sigma(1) \cdots \sigma(|\sigma|-1).
\end{align*}
In the case of the empty string $\epsilon$, $t \epsilon$ is $t$.  In this notation, for any closed terms $s$ and $t$ and any ordinal $\alpha$, we have that $s \sim_{\alpha + n} t$ if and only if $s \sigma \sim_\alpha t \sigma$ for every string $\sigma$ of length $n$.

\begin{lemma}\label{lem-EquivChar}
Let $s$ and $t$ be closed terms over $\K_1$.  Then $s \approx t$ if and only if for every sequence $a_0, a_1, a_2, \dots$ from $\K_1$, there is an $n$ such that $s a_0 a_1 \cdots a_n \simeq t a_0 a_1 \cdots a_n$.
\end{lemma}

\begin{proof}
Define a tree $T \subseteq \omega^{<\omega}$ by 
$T = \{\sigma \in \omega^{<\omega} : s \sigma \nsimeq t \sigma\}$.  
We show that $s \approx t$ if and only if $T$ is well-founded.  

Suppose that $s \approx t$.  Observe that for any closed terms $p$ and $q$, if $p \sim_\alpha q$ for some $\alpha > 0$, then for every $x \in \K_1$ there is a $\beta < \alpha$ such that $p x \sim_\beta q x$.  Now consider any $f \colon \omega \imp \omega$.  Define a sequence of ordinals $\alpha_0 \geq \alpha_1 \geq \cdots$, where $s f(0) \cdots f(n-1) \sim_{\alpha_n} t f(0) \cdots f(n-1)$ for each $n$, as follows.  First, let $\alpha_0$ be such that $s \sim_{\alpha_0} t$.  Then for each $n$, if $\alpha_n > 0$, let $\alpha_{n+1} < \alpha_n$ be such that $s f(0) \cdots f(n-1)f(n) \sim_{\alpha_{n+1}} t f(0) \cdots f(n-1)f(n)$.  If $\alpha_n = 0$, then let $\alpha_{n+1} = 0$.  The sequence $\alpha_0 \geq \alpha_1 \geq \cdots$ cannot be a strictly descending sequence of ordinals, so there must be an $n$ such that $\alpha_n = 0$.  Then for this $n$, $s f(0) \cdots f(n-1) \simeq t f(0) \cdots f(n-1)$, which means that $f$ is not a path through $T$.  Thus $T$ is well-founded.

For the converse, suppose that $T$ is well-founded.  Recall the \emph{Kleene--Brouwer ordering} of $\omega^{<\omega}$, where $\tau \ltKB \sigma$ if either $\tau$ is a proper extension of $\sigma$ or $\tau$ is to the left of $\sigma$.  That is, $\tau \ltKB \sigma$ if and only if
\begin{align*}
\tau \sqsupset \sigma \orr (\exists n < \min(|\sigma|, |\tau|))[\tau(n) < \sigma(n) \andd (\forall i < n)(\sigma(i) = \tau(i))].
\end{align*}
The tree $T$ is well-ordered by $\ltKB$ because we assume that $T$ is well-founded, and a subtree of $\omega^{<\omega}$ is well-founded if and only if it is well-ordered by $\ltKB$.  Let $\alpha$ be the ordinal isomorphic to $(T, \ltKB)$, and, for each $\sigma \in T$, let $\alpha_\sigma$ be the ordinal corresponding to $\sigma$.  We show by transfinite induction that $s \sigma \sim_{\alpha_\sigma + 1} t \sigma$ for every $\sigma \in T$.

Let $\sigma \in T$ and suppose inductively that $s \tau \sim_{\alpha_\tau + 1} t \tau$ for all $\tau \in T$ with $\tau \ltKB \sigma$.  Consider any $x$.  If $\sigma^\smf x \notin T$, then $s \cdot \sigma^\smf x \simeq t \cdot \sigma^\smf x$, so $s \cdot \sigma^\smf x \sim_{\alpha_\sigma} t \cdot \sigma^\smf x$.  If $\sigma^\smf x \in T$, then $\sigma^\smf x \ltKB \sigma$, so $s \cdot \sigma^\smf x \sim_{(\alpha_{\sigma^\smf x} + 1)} t \cdot \sigma^\smf x$.  Thus  $s \cdot \sigma^\smf x \sim_{\alpha_\sigma} t \cdot \sigma^\smf x$ because $\alpha_{\sigma^\smf x} < \alpha_\sigma$.  We have shown that $s \cdot \sigma^\smf x \sim_{\alpha_\sigma} t \cdot \sigma^\smf x$ for every $x$.  Therefore $s \sigma \sim_{\alpha_\sigma + 1} t \sigma$.

In the case of the empty string $\epsilon$, we have that $s \epsilon = s$ and $t \epsilon = t$ and therefore that $s \sim_{\alpha_\epsilon + 1} t$.  We may also observe that $\epsilon$ is the $\ltKB$-maximum element of $T$ and therefore that $\alpha = \alpha_\epsilon + 1$.  
So $s \sim_\alpha t$.  Thus $s \approx t$.
\end{proof}

Notice that Lemma~\ref{lem-EquivChar} also implies that the $\approx$ relation is $\Pi^1_1$.  Moreover, the proof of Lemma~\ref{lem-EquivChar} gives another proof that $\ord(\K_1) = \omega_1^{CK}$ and therefore that ${\approx} = {\sim_\OCK}$.  Let $s$ and $t$ be two closed terms.  If $s \approx t$, then the tree $T$ from the proof of Lemma~\ref{lem-EquivChar} is well-founded, and $s \sim_\alpha t$ for the ordinal $\alpha$ isomorphic to $(T, \ltKB)$.  The tree $T$ is $\Delta^0_2$ and the $\ltKB$ relation is computable on $\omega^{<\omega}$, 
so $\alpha < \omega_1^{\emptyset'} = \omega_1^{CK}$, where the equality is by~\cite{Sacks}*{Corollary~II.7.4}.

\begin{theorem}\label{thm-EquivComplexity}
The relation $\approx$ on $\K_1$ is $\Pi^1_1$-complete.
\end{theorem}

\begin{proof}
The relation $\approx$ is $\Pi^1_1$ as discussed above.  We show that $\approx$ is $\Pi^1_1$-hard.

A typical $\Pi^1_1$-complete set is the set of indices of partial computable functions computing well-founded subtrees of $\omega^{<\omega}$.  Indeed, there is a uniformly computable sequence $(T_e)_{e \in \omega}$ of trees such that the set $\{e : \text{$T_e$ is well-founded}\}$ is $\Pi^1_1$-complete (see, for example,~\cite{CenzerRemmelIndex}*{Theorem~15}).

Using the recursion theorem, fix an index $e_* \in \K_1$ such that $e_* \cdot x = e_*$ for every $x$.

Let $g(i, e, \sigma)$ be a total computable function such that for all $i$, $e$, $\sigma$, and $x$,
\begin{align*}
\Phi_{g(i, e, \sigma)}(x) \simeq \Phi_i(e, \sigma^\smf x).
\end{align*}
By padding, we may define $g$ so that $e_* \notin \ran(g)$.  Using the recursion theorem, let $i$ be such that
\begin{align*}
\Phi_i(e, \sigma) =
\begin{cases}
e_* & \text{if $\sigma \notin T_e$}\\
g(i, e, \sigma) & \text{if $\sigma \in T_e$}.
\end{cases}
\end{align*}
Note that $\Phi_i$ is total because $g$ is total.  Also note that $\Phi_i(e, \sigma) = e_*$ if and only if $\sigma \notin T_e$ because $e_* \notin \ran(g)$.  Let $f$ be the function computed by $\Phi_i$.

\begin{claim*}
For every $e$, $\sigma$, and $x$, $f(e,\sigma) \cdot x = f(e, \sigma^\smf x)$.
\end{claim*}

\begin{proof}[Proof of Claim]
If $\sigma \notin T_e$, then $\sigma^\smf x \notin T_e$ as well, so $f(e, \sigma) \cdot x = e_* \cdot x = e_* = f(e, \sigma^\smf x)$.  If $\sigma \in T_e$, then
\begin{align*}
f(e,\sigma) \cdot x = \Phi_i(e, \sigma) \cdot x = g(i,e,\sigma) \cdot x = \Phi_{g(i,e,\sigma)}(x) = \Phi_i(e, \sigma^\smf x) = f(e, \sigma^\smf x).
\end{align*}
\end{proof}

We now show that, for every $e$, $f(e, \epsilon) \approx e_*$ if and only if $T_e$ is well-founded.  Notice that $e_* a_0 a_1 \cdots a_n = e_*$ for every sequence $a_0, a_1 \dots, a_n$.  Therefore, by Lemma~\ref{lem-EquivChar}, $f(e, \epsilon) \approx e_*$ if and only if for every sequence $a_0, a_1, a_2 \dots$, there is an $n$ such that $f(e, \epsilon) \cdot a_0a_1 \cdots a_n = e_*$.

Suppose that $T_e$ is well-founded.  Consider any sequence $a_0, a_1, a_2, \dots$.  This sequence is not a path though $T_e$, so there is an $n$ such that the string $\sigma = \la a_0, \dots, a_n \ra$ is not in $T_e$.  By the Claim and the fact that $\sigma \notin T_e$, we have that $f(e, \epsilon) \cdot a_0a_1 \cdots a_n = f(e, \sigma) = e_*$.  Thus $f(e, \epsilon) \approx e_*$.

Conversely, suppose that $T_e$ is ill-founded, and let $a_0, a_1, a_2, \dots$ be a path through $T_e$.  For each $n$, let $\sigma_n = \la a_0, \dots, a_n \ra$.  By the Claim, for each $n$ we have that $f(e, \epsilon) \cdot a_0a_1 \cdots a_n = f(e, \sigma_n) \neq e_*$, where the inequality is because $\sigma_n \in T_e$.  Thus $f(e, \epsilon) \napprox e_*$.

We have shown that $f(e, \epsilon) \approx e_*$ if and only if $T_e$ is well-founded.  The map $e \mapsto (f(e, \epsilon), e_*)$ is therefore a many-one reduction witnessing that $\{e : \text{$T_e$ is well-founded}\} \leqm {\approx}$.  Thus the $\approx$ relation is $\Pi^1_1$-hard and hence $\Pi^1_1$-complete.
\end{proof}

Notice that the many-one reduction in the proof of Theorem~\ref{thm-EquivComplexity} always produces elements of $\K_1$, as opposed to a more general closed terms.  Thus the $\approx$ relation is $\Pi^1_1$-complete when thought of either as a relation on $\K_1$ or as a relation on the closed terms over $\K_1$.

For every $X \subseteq \omega$ we have the relativized pca 
$\K_1^X$, with application $n\cdot m = \Phi^X_n(m)$. 
By relativizing Theorem~\ref{thm:K1}, we obtain 
$\ord(\K_1^X) = \omega_1^X$ (i.e.\ $\OCK$ relative to $X$). 
Note that $\omega_1^X = \OCK$ for almost every $X$, in the sense of measure on $2^\omega$ (cf.\ \cite{Sacks}*{Corollary~IV.1.6}).
More generally, if $\A$ is any countable pca, then the domain of $\A$ may be taken to be $\omega$, in which case $\A$'s partial application operation is partial computable relative to some set $X \subseteq \omega$.  For example, $X$ may be taken to be the graph $X = \{\la a, b, c \ra :  a \cdot b = c\}$ of $\A$'s partial application operation.  In the analysis given in the proof of Theorem~\ref{thm:upper} (and of Lemma~\ref{lem-EquivChar}), we may replace $\K_1$ by $\A$ and relativize to $X$ to obtain the following.

\begin{theorem}\label{thm:countable}
Let $\A$ be a countable pca.  If $\A$'s partial application operation is partial computable relative to $X$, then $\ord(\A) \leq \omega_1^X$.  In particular, $\ord(\A)$ is countable.
\end{theorem}

\section{Arithmetical complexity in $\K_1$} \label{sec:arithmetical}

We work only with $\K_1$ in this section and the next.
In the proof of Theorem~\ref{thm:upper} we observed that 
the relation $s \sim_\alpha t$ on the closed terms over $\K_1$ is 
(uniformly) hyperarithmetical for every $\alpha <\OCK$.
Writing out the definitions of the first levels of the hierarchy, 
we see the following:\footnote{Additionally, $\sim_0$ is decidable when restricted to $\K_1$.}
\begin{align*}
\sim_0 &\text{ is } \Delta^0_2 \\
\sim_n &\text{ is } \Pi^0_2 \text{ for every $n > 0$} \\
\sim_\omega &\text{ is } \Sigma^0_3 \\  
\sim_{\omega + n} &\text{ is } \Pi^0_4 \text{ for every $n > 0$} \\
\sim_{\omega 2} &\text{ is } \Sigma^0_5 \\
&\vdots \\
\sim_{\omega k} &\text{ is } \Sigma^0_{2k+1} \text{ for every $k > 0$} \\
\sim_{\omega k + n} &\text{ is } \Pi^0_{2k+2} \text{ for every $k$ and $n > 0$}.
\end{align*}
See Lemma~\ref{lem-HypDefF} below for a proof that encompasses the full hyperarithmetical hierarchy.

We now verify that $\sim_\alpha$ is complete at the indicated level of the arithmetical hierarchy for all $0 < \alpha < \omega^2$.  In Section~\ref{sec:hyp}, we show that this pattern extends to the full hyperarithmetical hierarchy.  We use different proofs for the arithmetical levels and the non-arithmetical levels.  The hardness proof for the arithmetical levels given in this section does not seem to easily generalize to the $\sim_{\omega^2}$ / $\Sigma^0_\omega$ level and beyond.  On the other hand, the hardness proof in Section~\ref{sec:hyp} is off-by-one at the arithmetical levels.  The strategy of Section~\ref{sec:hyp} restricts to hereditarily total elements of $\K_1$.  Such a strategy cannot work at the arithmetical levels because, for example, $\sim_1$ is $\Pi^0_1$ when restricted to total elements.  Thus strategy of Section~\ref{sec:hyp} only shows that $\sim_1$ is $\Pi^0_1$-hard, whereas here we show that $\sim_1$ is $\Pi^0_2$-hard.  This off-by-one discrepancy catches up at the $\sim_{\omega^2}$ / $\Sigma^0_\omega$ level and thus yields the desired completeness at the non-arithmetical levels.

As with the proof of Theorem~\ref{thm-EquivComplexity}, our many-one reductions to the $\sim_\alpha$ relations always produce elements of $\K_1$, as opposed to more general closed terms.  Thus the completeness results hold when the $\sim_\alpha$ relations are thought of either as relations on $\K_1$ or as relations on the closed terms over $\K_1$.

In this section, all formulas are finitary.

\begin{lemma}\label{lem-ReduceBy1}
For every ordinal $\alpha$, ${\sim_{\alpha + 1}} \equivm {\sim_{\alpha + 2}}$.
\end{lemma}

\begin{proof}
We first show that ${\sim_{\alpha + 1}} \leqm {\sim_{\alpha + 2}}$.  Given closed terms $s$ and $t$, we can effectively produce indices $f, g \in \K_1$ where $f \sim_1 K s$ and $g \sim_1 K t$.  Then $f \sim_{\alpha + 2} g$ if and only if $s \sim_{\alpha + 1} t$ by Lemma~\ref{lem-AddK} and because $\alpha + 1 \geq 1$.  Thus the map $(s, t) \mapsto (f, g)$ is a many-one reduction witnessing that ${\sim_{\alpha + 1}} \leqm {\sim_{\alpha + 2}}$.

Now we show that ${\sim_{\alpha + 2}} \leqm {\sim_{\alpha + 1}}$.  For this, let $\la \cdot, \cdot \ra \colon \omega \times \omega \imp \omega$ denote a computable bijection.  Given closed terms $s$ and $t$, we can effectively produce indices $f, g \in \K_1$ such that $f \cdot \la a, b \ra \simeq s a b$ and $g \cdot \la a, b \ra \simeq t a b$ for all $a$ and $b$.  The map $(s, t) \mapsto (f, g)$ is then a many-one reduction witnessing that ${\sim_{\alpha + 2}} \leqm {\sim_{\alpha + 1}}$.  To see this, first suppose that $s \sim_{\alpha + 2} t$.  Then $f \cdot \la a, b \ra \sim_0 s a b \sim_\alpha t a b \sim_0 g \cdot \la a, b \ra$ for all $a$ and $b$.  Thus $f \cdot \la a, b \ra \sim_\alpha g \cdot \la a, b \ra$ for all $\la a, b \ra$, so $f \sim_{\alpha + 1} g$.  Conversely, suppose that $s \not\sim_{\alpha + 2} t$.  Then there are $a$ and $b$ such that $f \cdot \la a, b \ra \sim_0 s a b \not\sim_\alpha t a b \sim_0 g \cdot \la a, b \ra$.  Therefore there is an $\la a, b \ra$ such that 
$f \cdot \la a, b \ra \not\sim_\alpha g \cdot \la a, b \ra$, 
so $f \not\sim_{\alpha + 1} g$.
\end{proof}

\begin{lemma}\label{lem-HardPi2}
The relation $\sim_1$ is $\Pi^0_2$-hard.
\end{lemma}

\begin{proof}
Let $\forall n \exists m \; \psi(z, n, m)$ be a $\Pi^0_2$ formula, where $\psi$ is $\Sigma^0_0$.  We show that, given $z$ and $\ell \geq 1$, we can effectively produce indices $f$ and $g$ such that
\begin{itemize}
\item $f \sim_1 g$ if $\forall n \exists m \; \psi(z, n, m)$ holds,
\item $f \not\sim_\ell g$ if $\forall n \exists m \; \psi(z, n, m)$ fails, and
\item $f \sim_{\ell + 1} g$ regardless of whether or not $\forall n \exists m \; \psi(z, n, m)$ holds.
\end{itemize}

Fix, say, $\ell = 1$.  Then the map $z \mapsto (f, g)$ witnesses that $\{z : \forall n \exists m \; \psi(z,n,m)\} \leqm {\sim_1}$.  It follows that $\sim_1$ is $\Pi^0_2$-hard.  All that is required for this many-one reduction is that $f \sim_1 g$ if and only if $\forall n \exists m \; \psi(z, n, m)$ holds.  The extra features are useful for the proof of Lemma~\ref{lem-HardOmegaSigma3}.

Fix an index $c$ such that $c x \ua$ for all $x$.  Define indices $d_\ell$ for each $\ell \geq 0$ by $d_0 = c$ and $d_{\ell + 1} = K d_\ell$.  Then for every $\ell$ and every $x_0, \dots, x_\ell$, we have that $d_\ell x_0 \cdots x_{\ell - 1}\da = c$ and that $d_\ell x_0 \cdots x_{\ell - 1} x_\ell\ua$.

Given $z$ and $\ell \geq 1$, let $g = d_\ell$ and compute an index $f$ such that
\begin{align*}
\Phi_f(n) \simeq 
\begin{cases}
d_{\ell-1} & \text{if $\exists m \; \psi(z, n, m)$}\\
\ua & \text{otherwise}.
\end{cases}
\end{align*}
Notice that $f x_0 \cdots x_\ell \ua$ and $g x_0 \cdots x_\ell \ua$ for all $x_0, \dots, x_\ell$, so $f \sim_{\ell + 1} g$.  Suppose that $\forall n \exists m \; \psi(z, n, m)$ holds.  Then $f n = d_{\ell-1} = g n$ for every $n$, so $f \sim_1 g$.  Conversely, suppose that $\exists m \; \psi(z, n, m)$ fails for some $n$.  Let $x_1 = \cdots = x_{\ell-1} = 0$.  Then $f n \ua$ and so $f n x_1 \cdots x_{\ell-1}\ua$, but $g n x_1 \cdots x_{\ell-1}\da = c$.  Thus $f \not\sim_\ell g$.  This completes the proof.
\end{proof}

By combining Lemma~\ref{lem-ReduceBy1} and Lemma~\ref{lem-HardPi2}, we see that $\sim_n$ is $\Pi^0_2$-hard for each $n > 0$.  Now we show that $\sim_\omega$ is $\Sigma^0_3$-hard.

\begin{lemma}\label{lem-HardOmegaSigma3}
The relation $\sim_\omega$ is $\Sigma^0_3$-hard.
\end{lemma}

\begin{proof}
Let $\exists n \; \psi(z, n)$ be an arbitrary $\Sigma^0_3$ formula, where $\psi$ is $\Pi^0_2$.  We may assume that $\psi(z, n) \imp \psi(z, n+1)$ for every $z$ and $n$ by replacing $\exists n \; \psi(z,n)$ by the equivalent $(\exists n)(\exists p < n)\psi(z, p)$ and pushing the bounded quantifier past the unbounded quantifiers.

We show that, given $z$, we can effectively produce indices $f, g \in \K_1$ such that
\begin{itemize}
\item $f \sim_\omega g$ if and only if $\exists n \; \psi(z, n)$ holds and
\item $f \sim_{\omega + 1} g$ regardless of whether or not $\exists n \; \psi(z, n)$ holds.
\end{itemize}

Given such a procedure, the map $z \mapsto (f, g)$ witnesses that $\{z : \exists n \; \psi(z, n)\} \leqm {\sim_\omega}$.  It follows that $\sim_\omega$ is $\Sigma^0_3$-hard.  The additional $f \sim_{\omega + 1} g$ requirement is not needed for the many-one reduction, but it is helpful for the proof of Lemma~\ref{lem-HardOmegaKSigma2K1}.

By the proof of Lemma~\ref{lem-HardPi2}, given $z$, $n$, and $\ell \geq 1$, we can effectively produce indices $a^{z, \ell}_n$ and $b^{z, \ell}_n$ such that
\begin{itemize}
\item $a^{z, \ell}_n \sim_1 b^{z, \ell}_n$ if $\psi(z, n)$ holds,

\smallskip

\item $a^{z, \ell}_n \not\sim_\ell b^{z, \ell}_n$ if $\psi(z, n)$ fails, and

\smallskip

\item $a^{z, \ell}_n \sim_{\ell + 1} b^{z, \ell}_n$ regardless of whether or not $\psi(z, n)$ holds.
\end{itemize}

Now, given $z$, define $f$ and $g$ so that $f n = a^{z, n+1}_n$ and $g n = b^{z, n+1}_n$ for each $n$.

To see that $f \sim_{\omega + 1} g$, observe that $f n = a^{z, n+1}_n \sim_{n+2} b^{z, n+1}_n  = g n$ for every $n$.  Thus $f n \sim_\omega g n$ for every $n$, so $f \sim_{\omega + 1} g$.

Suppose that there is an $n_0$ such that $\psi(z, n_0)$ holds.  We show that $f n \sim_{n_0 + 1} g n$ for every $n$.  Therefore $f \sim_{n_0 + 2} g$, so $f \sim_\omega g$.  Fix $n$.  If $n < n_0$, then $f n \sim_{n_0 + 1} g n$ because $f n \sim_{n+2} g n$ as above, and $n + 2 \leq n_0 + 1$.  Suppose instead that $n \geq n_0$.  Then $\psi(z, n)$ holds because $\psi(z, n_0)$ holds and $n \geq n_0$.  Thus $f n = a^{z, n+1}_n \sim_1 b^{z, n+1}_n = g n$.  That is, $f n \sim_1 g n$, so also $f n \sim_{n_0 + 1} g n$.

Finally, suppose that $\psi(z, n)$ fails for every $n$.  Then 
$f n = a^{z, n+1}_n \not\sim_{n + 1} b^{z, n+1}_n = g n$ for every $n$.  Thus also 
$f \not\sim_{n + 1} g$ for every $n$, so $f \not\sim_\omega g$.  This completes the proof.
\end{proof}

The following lemma helps us generalize Lemma~\ref{lem-HardOmegaSigma3} to the $\sim_{\omega k}$ relations for all $k \geq 1$.  In Lemma~\ref{lem-ArithHardHelper} and Lemma~\ref{lem-HardOmegaKSigma2K1} we distinguish between a computable bijective encoding of pairs $\la \cdot, \cdot \ra_2 \colon \omega \times \omega \imp \omega$ and a computable bijective encoding of all finite strings $\la \cdot \ra \colon \omega^{<\omega} \imp \omega$.

\begin{lemma}\label{lem-ArithHardHelper}
For each $n, m \in \omega$, let $\sigma_{n,m} = {\la n, m \ra_2}^\smf 0^n$ (i.e., the string of length $n+1$ with $\sigma_{n,m}(0) = \la n, m \ra_2$ and $\sigma_{n,m}(i) = 0$ for all $0 < i \leq n$).  Given an index $e$ for a computable sequence $(a_{n,m})_{n,m \in \omega}$, we can effectively produce an $f \in \K_1$ such that
\begin{itemize}
\item $f \sigma_{n,m} = a_{n,m}$ for all $n, m \in \omega$, and
\item $f \tau \ua$ if $\tau$ is $\sqsubseteq$-incomparable with $\sigma_{n,m}$ for all $n, m \in \omega$.
\end{itemize}
\end{lemma}

\begin{proof}
The proof is similar to that of Theorem~\ref{thm-EquivComplexity}.  Let $g(i, e, \sigma)$ be a total computable function such that for all $i$, $e$, $\sigma$, and $x$,
\begin{align*}
\Phi_{g(i, e, \sigma)}(x) \simeq \Phi_i(e, \sigma^\smf x).
\end{align*}
Using the recursion theorem, let $i$ be such that
\begin{align*}
\Phi_i(e, \sigma) \simeq
\begin{cases}
\Phi_e(n,m) & \text{if $\sigma = \sigma_{n,m}$}\\
g(i, e, \sigma) & \text{if $\sigma \sqsubset \sigma_{n,m}$ for some $n$ and $m$}\\
\ua & \text{otherwise.}
\end{cases}
\end{align*}
Let $h$ be the function partially computed by $\Phi_i$.  Let $e$ be an index for the sequence $(a_{n,m})_{n,m \in \omega}$, so that $\Phi_e(n,m) = a_{n,m}$ for all $n$ and $m$.

\begin{claim*}
For every $e$, $\sigma$, and $x$, if $\sigma \sqsubset \sigma_{n,m}$ for some $n$ and $m$, then $h(e, \sigma) \cdot x \simeq h(e, \sigma^\smf x)$.
\end{claim*}

\begin{proof}[Proof of Claim]
If $\sigma \sqsubset \sigma_{n,m}$, then 
\begin{align*}
h(e, \sigma) \cdot x \simeq \Phi_i(e, \sigma) \cdot x \simeq g(i,e,\sigma) \cdot x \simeq \Phi_{g(i,e,\sigma)}(x) \simeq \Phi_i(e, \sigma^\smf x) \simeq h(e, \sigma^\smf x).
\end{align*}
\end{proof}

Take $f = h(e, \epsilon) = g(i,e,\epsilon)$, which is defined because $g$ is total.  By the Claim, for each $n$ and $m$ we have that $f \sigma_{n,m} \simeq h(e, \epsilon) \cdot \sigma_{n,m} \simeq h(e, \sigma_{n,m}) \simeq \Phi_e(n,m) = a_{n,m}$.  Thus $f \sigma_{n,m} = a_{n,m}$ for all $n$ and $m$.

Suppose that $\tau$ is $\sqsubseteq$-incomparable with $\sigma_{n,m}$ for all $n$ and $m$.  Let $\sigma \sqsubset \tau$ be the longest initial segment of $\tau$ that is $\sqsubseteq$-comparable with $\sigma_{n_0, m_0}$ for some $n_0$ and $m_0$.  Then $\sigma \sqsubset \sigma_{n_0, m_0}$, but $\sigma^\smf \tau(|\sigma|) = \tau \restr (|\sigma|+1)$ is $\sqsubseteq$-incomparable with $\sigma_{n,m}$ for all $n$ and $m$.  By the Claim, we have that
\begin{align*}
f \cdot \tau \restr (|\sigma|+1) \simeq f \cdot \sigma^\smf \tau(|\sigma|) \simeq h(e, \epsilon) \cdot \sigma^\smf \tau(|\sigma|) \simeq h(e, \sigma^\smf \tau(|\sigma|)).
\end{align*}
However, $h(e, \sigma^\smf \tau(|\sigma|))\ua$ because $\sigma^\smf \tau(|\sigma|)$ is $\sqsubseteq$-incomparable with $\sigma_{n,m}$ for all $n$ and $m$.  Therefore $(f \cdot \tau \restr (|\sigma|+1))\ua$.  Thus $f \tau \ua$ as well.  Thus $f \tau \ua$ whenever $\tau$ is $\sqsubseteq$-incomparable with $\sigma_{n,m}$ for all $n$ and $m$.
\end{proof}

\begin{lemma}\label{lem-HardOmegaKSigma2K1}
The relation $\sim_{\omega k}$ is $\Sigma^0_{2k+1}$-hard for every $k > 0$.
\end{lemma} 

\begin{proof}
We induct on $k > 0$ to show the following.  Given a $\Sigma^0_{2k+1}$ formula $\exists n \forall m \; \psi(z, n, m)$, where $\psi$ is $\Sigma^0_{2k-1}$, and a $z$, we can effectively produce indices $f, g \in \K_1$ such that
\begin{itemize}
\item $f \sim_{\omega k} g$ if and only if $\exists n \forall m \; \psi(z, n, m)$ holds and
\item $f \sim_{\omega k + 1} g$ regardless of whether or not $\exists n \forall m \; \psi(z, n, m)$ holds.
\end{itemize}

Given such a procedure, the map $z \mapsto (f, g)$ witnesses that $\{z : \exists n \forall m \; \psi(z,n,m)\} \leqm {\sim_{\omega k}}$.  It follows that $\sim_{\omega k}$ is $\Sigma^0_{2k+1}$-hard.

The base case $k = 1$ is by Lemma~\ref{lem-HardOmegaSigma3} and its proof.  Thus suppose that $k > 1$, and consider an arbitrary $\Sigma^0_{2k+1}$ formula $\exists n \forall m \; \psi(z, n, m)$, where $\psi(z, n, m)$ is $\Sigma^0_{2k-1}$.  We may assume that $\forall m \; \psi(z, n, m) \imp \forall m \; \psi(z, n+1, m)$ for every $z$ and $n$ by replacing $\exists n \forall m \; \psi(z,n,m)$ by the equivalent $(\exists n)(\exists p < n)(\forall m)\psi(z, p, m)$ and pushing the bounded quantifier past the unbounded quantifiers.

Inductively assume that, given $z$, $n$, and $m$, we can effectively produce indices $a^z_{n, m}$ and $b^z_{n, m}$ such that

\begin{itemize}
\item $a^z_{n, m} \sim_{\omega (k-1)} b^z_{n, m}$ if and only if $\psi(z,n,m)$ holds, and
\item $a^z_{n, m} \sim_{\omega (k-1) + 1} b^z_{n, m}$ regardless of whether or not $\psi(z,n,m)$ holds.
\end{itemize}

Let $\sigma_{n,m} = {\la n, m \ra_2}^\smf 0^n$ for each $n, m \in \omega$ as in Lemma~\ref{lem-ArithHardHelper}.  Now, given $z$, use Lemma~\ref{lem-ArithHardHelper} to effectively produce indices $f, g \in \K_1$ so that
\begin{itemize}
\item $f \sigma_{n,m} = a^z_{n, m}$ and $g \sigma_{n,m} = b^z_{n, m}$ for all $n, m \in \omega$, and
\item $f \tau \ua$ and $g \tau \ua$ if $\tau$ is $\sqsubseteq$-incomparable with $\sigma_{n,m}$ for all $n, m \in \omega$.
\end{itemize}

First, we show that $f \sim_{\omega k + 1} g$.  Consider any $x$, and decode $x$ as $x = \la n, m \ra_2$.  Now consider any $\tau$ of length $n$.  If $\tau = 0^n$, then $x^\smf \tau = \sigma_{n,m}$, in which case
\begin{align*}
f \cdot x^\smf\tau = f \sigma_{n,m} = a^z_{n, m} \sim_{\omega(k-1) + 1} b^z_{n, m} = g \sigma_{n,m} = g \cdot x^\smf\tau.
\end{align*}
If $\tau \neq 0^n$, then $x^\smf \tau$ is $\sqsubseteq$-incomparable with $\sigma_{n',m'}$ for all $n'$ and $m'$, so $(f \cdot x^\smf\tau)\ua$ and $(g \cdot x^\smf\tau)\ua$, which means that $f \cdot x^\smf\tau \simeq g \cdot x^\smf\tau$ and hence that $f \cdot x^\smf\tau \sim_{\omega(k-1) + 1} g \cdot x^\smf\tau$.  We therefore have that $f \cdot x^\smf\tau \sim_{\omega(k-1) + 1} g \cdot x^\smf\tau$ for all $\tau$ of length $n$.  This means that $f x \sim_{\omega(k-1) + n + 1} g x$ and therefore that $f x \sim_{\omega k} g x$.  So $f x \sim_{\omega k} g x$ for every $x$.  Thus $f \sim_{\omega k + 1} g$.

Now suppose that there is an $n_0$ such that $\forall m \; \psi(z, n_0, m)$ holds.  In this case, we show that $f \tau \sim_{\omega(k-1)} g \tau$ for every $\tau$ of length $n_0 + 1$.  It then follows that $f \sim_{\omega(k-1) + n_0 + 1} g$ and hence that $f \sim_{\omega k} g$.  So consider a $\tau$ of length $n_0 + 1$.  If $\tau$ is $\sqsubseteq$-incomparable with $\sigma_{n, m}$ for all $n$ and $m$, then $f \tau \ua$ and $g \tau \ua$, so $f \tau \simeq g \tau$, so $f \tau \sim_{\omega(k-1)} g \tau$.  Suppose instead that $\sigma_{n, m} \sqsubset \tau$ for some $n$ and $m$.  Then $\tau = {\sigma_{n,m}}^\smf \eta$ for some string $\eta$ of length $\ell > 0$.  So $f \tau = f \cdot {\sigma_{n,m}}^\smf \eta \simeq a^z_{n, m} \cdot \eta$, and $g \tau = g \cdot {\sigma_{n,m}}^\smf \eta \simeq b^z_{n, m} \cdot \eta$.  We know that $a^z_{n, m} \sim_{\omega (k-1) + 1} b^z_{n, m}$, so $a^z_{n, m} \cdot \eta \sim_{\omega (k-1)} b^z_{n, m} \cdot \eta$ because $|\eta| > 0$.  All together, we therefore have that $f \tau \sim_{\omega (k-1)} g \tau$.  Finally, suppose that $\tau \sqsubseteq \sigma_{n, m}$ for some $n$ and $m$.  It must then be that $n_0 + 1 = |\tau| \leq |\sigma_{n,m}| = n + 1$, so $n_0 \leq n$.  Therefore $\forall m \; \psi(z, n, m)$ holds because $\forall m \; \psi(z, n_0, m)$ holds and $n_0 \leq n$.  So,
\begin{align*}
f \sigma_{n,m} = a^z_{n, m} \sim_{\omega (k-1)} b^z_{n, m} =  g \sigma_{n,m}.
\end{align*}
Thus there is an $\alpha < \omega (k-1)$ such that $f \sigma_{n,m} \sim_\alpha g \sigma_{n,m}$.  Now consider any $\eta$ of length $\ell = n - n_0$.  If $\eta = 0^\ell$, then $\tau^\smf \eta = \sigma_{n,m}$, so $f \cdot \tau^\smf \eta \sim_\alpha g \cdot \tau^\smf \eta$.  If $\eta \neq 0^\ell$, then $\tau^\smf\eta$ is $\sqsubseteq$-incomparable with $\sigma_{n', m'}$ for all $n'$ and $m'$, so $(f \cdot \tau^\smf\eta) \ua \simeq (g \cdot \tau^\smf\eta) \ua$.  Thus also $f \cdot \tau^\smf\eta \sim_\alpha g \cdot \tau^\smf\eta$.  So $f \cdot \tau^\smf\eta \sim_\alpha g \cdot \tau^\smf\eta$ for every $\eta$ of length $\ell$.  This means that $f \tau \sim_{\alpha + \ell} g \tau$.  We have that $\alpha + \ell < \omega (k-1)$ because $\alpha < \omega (k-1)$ and $\omega (k-1)$ is a limit.  Thus $f \tau \sim_{\omega (k-1)} g \tau$.

Finally, suppose that $\forall m \; \psi(z, n, m)$ fails for every $n$.  We show that $f \not\sim_{\omega k} g$.  Given $n$, let $m$ be such that $\psi(z, n, m)$ fails, and consider $\sigma_{n, m}$.  Then $f \sigma_{n,m} = a^z_{n, m} \not\sim_{\omega(k-1)} b^z_{n, m} = g \sigma_{n,m}$.  As $|\sigma_{n,m}| = n+1$, this implies that $f \not\sim_{\omega(k-1) + n + 1} g$.  So $f \not\sim_{\omega(k-1) + n + 1} g$ for every $n$.  Thus $f \not\sim_{\omega k} g$.

Given $z$, we have effectively produced indices $f$ and $g$ such that
\begin{itemize}
\item $f \sim_{\omega k} g$ if and only if $\exists n \forall m \; \psi(z,n,m)$ holds, and
\item $f \sim_{\omega k + 1} g$ regardless of whether or not $\exists n \forall m \; \psi(z,n,m)$ holds.
\end{itemize}
This completes the induction step and therefore completes the proof.
\end{proof}

\begin{lemma}\label{lem-HardOmegaK1Pi2K2}
The relation $\sim_{\omega k + 1}$ is $\Pi^0_{2k + 2}$-hard for every $k$.
\end{lemma}

\begin{proof}
For $k = 0$, this is given by Lemma~\ref{lem-HardPi2}.  Let $k > 0$, and consider a $\Pi^0_{2k+2}$ formula $\forall n \; \psi(z, n)$, where $\psi(z,n)$ is $\Sigma^0_{2k+1}$.  By Lemma~\ref{lem-HardOmegaKSigma2K1} and its proof, given $z$ and $n$, we can effectively produce indices $a^z_n, b^z_n \in \K_1$ such that $a^z_n \sim_{\omega k} b^z_n$ if and only if $\psi(z, n)$ holds.  Thus given $z$, we can effectively produce indices $f, g \in \K_1$ where $f n = a^z_n$ and $g n = b^z_n$ for all $n \in \omega$.

If $\forall n \; \psi(z, n)$ holds, then $f n = a^z_n \sim_{\omega k} b^z_n = g n$ for every $n$.  Thus $f \sim_{\omega k + 1} g$.  On the other hand, if $n$ is such that $\psi(z, n)$ fails, then $f n = a^z_n \not\sim_{\omega k} b^z_n = g n$, so $f \not\sim_{\omega k + 1} g$.  Therefore the map $z \mapsto (f, g)$ witnesses that $\{z : \forall n \; \psi(z, n)\} \leqm {\sim_{\omega k + 1}}$.  Thus $\sim_{\omega k + 1}$ is $\Pi^0_{2k + 2}$-hard.
\end{proof}

\begin{theorem}\label{thm:ArithComplexity}
{\ }
\begin{itemize}
\item The relation $\sim_{\omega k}$ is $\Sigma^0_{2k + 1}$-complete for every $k > 0$.
\item The relation $\sim_{\omega k + n}$ is $\Pi^0_{2k + 2}$-complete for every $k$ and every $n > 0$.
\end{itemize}
\end{theorem}

\begin{proof}
It is straightforward to check that $\sim_{\omega k}$ is $\Sigma^0_{2k + 1}$-definable for every $k > 0$ and that $\sim_{\omega k + n}$ is $\Pi^0_{2k + 2}$-definable for every $k$ and every $n > 0$.  See also Lemma~\ref{lem-HypDefF} below.  The relation $\sim_{\omega k}$ is $\Sigma^0_{2k + 1}$-hard for every $k > 0$ by Lemma~\ref{lem-HardOmegaKSigma2K1}.  The relation $\sim_{\omega k + 1}$ is $\Pi^0_{2k + 2}$-hard for every $k$ by Lemma~\ref{lem-HardOmegaK1Pi2K2}.  For $n > 0$, we have that ${\sim_{\omega k + n}} \equivm {\sim_{\omega k + 1}}$ by repeated applications of Lemma~\ref{lem-ReduceBy1}, so $\sim_{\omega k + n}$ is also $\Pi^0_{2k + 2}$-hard for every $k$ and every $n > 0$.
\end{proof}

\section{Hyperarithmetical complexity in $\K_1$} \label{sec:hyp}

In this section, we continue working with $\K_1$ and show that the complexity pattern described at the beginning of Section~\ref{sec:arithmetical} extends to the whole hyperarithmetical hierarchy.  To describe the pattern, we introduce two functions $F$ and $G$ on the ordinals.

\begin{definition}
Define functions $F$ and $G$ on the ordinals by
\begin{align*}
F(0) &= 0 & \\ 
F(\alpha + d + 1) &= F(\alpha) + 1 & &\text{if $d < \omega$ and either $\alpha = 0$ or $\alpha$ is a limit}\\
F(\alpha) &= \lim_{\beta < \alpha}(F(\beta) + 1)& &\text{if $\alpha$ is a limit}.\\
\\
G(0) &= 0 & \\
G(\alpha + 2d + 1) &= G(\alpha + 2d) + 1 & &\text{if $d < \omega$ and either $\alpha = 0$ or $\alpha$ is a limit}\\
G(\alpha + 2d + 2) &= G(\alpha + 2d + 1) + \omega & &\text{if $d < \omega$ and either $\alpha = 0$ or $\alpha$ is a limit}\\
G(\alpha) &= \lim_{\beta < \alpha}G(\beta) & &\text{if $\alpha$ is a limit}.\\
\end{align*}
\end{definition}

Ultimately, we show that the $\sim_\alpha$ relation is $\Pi^0_{1 + F(\alpha)}$-complete for every computable successor ordinal $\alpha \geq 1$ and is $\Sigma^0_{1+F(\alpha)}$-complete for every computable limit ordinal $\alpha$.  For ordinals $\alpha < \omega^2$, we calculate
\begin{align*}
1 + F(1) &= 2 \\
1 + F(\omega) &= 3 \\
1 + F(\omega + n) &= 4 \text{ for every $n  > 0$} \\
1 + F(\omega 2) &= 5 \\
&\vdots \\
1 + F(\omega k) &= 2k+1 \text{ for every $k > 0$} \\
1 + F(\omega k + n) &= 2k+2 \text{ for every $k$ and $n > 0$}.
\end{align*}
So Theorem~\ref{thm:ArithComplexity} already gives the desired result for ordinals $\alpha < \omega^2$.

We collect some helpful facts about $F$ and $G$.  First, it is easy to check that if $\beta < \alpha$, then $F(\beta) \leq F(\alpha)$ and $G(\beta) < G(\alpha)$.  Second, the range of $G$ consists of $0$, $1$, all limit ordinals, and all successors of limit ordinals.  Third, $F(G(\alpha)) = \alpha$ for every ordinal $\alpha$.

\begin{lemma}\label{lem-RanG}
The range of $G$ consists of all ordinals of the form $\gamma$ and $\gamma + 1$, where $\gamma$ is either $0$ or a limit ordinal.
\end{lemma}

\begin{proof}
Clearly $G(0) = 0$ and $G(1) = 1$.  A straightforward argument shows that $G(\alpha)$ is a limit ordinal when $\alpha > 0$ is even and that $G(\alpha)$ is the successor of a limit ordinal when $\alpha > 1$ is odd.

To see that the range of $G$ contains all limit ordinals, suppose for a contradiction that some limit ordinal is missing, and let $\gamma$ be the least limit ordinal not in the range of $G$.  First suppose that there is a maximum limit ordinal $\delta < \gamma$, in which case it must be that $\gamma = \delta + \omega$.  As $\delta < \gamma$, there is an $\alpha$ such that $G(\alpha) = \delta$, and $\alpha$ must be even by the discussion above.  Thus $G(\alpha + 2) = \delta + \omega = \gamma$, which is a contradiction.  Now suppose instead that there is no maximum limit ordinal $\delta < \gamma$.  Then $\gamma = \sup\{\delta < \gamma : \text{$\delta$ is a limit}\}$.  For each limit ordinal $\delta < \gamma$, let $\alpha_\delta$ be such that $G(\alpha_\delta) = \delta$, and let $\alpha = \sup\{\alpha_\delta : \text{$\delta < \gamma$ is a limit}\}$.  Then $G(\alpha) = \sup\{G(\alpha_\delta) : \text{$\delta < \gamma$ is a limit}\} = \gamma$, which is again a contradiction.  So the range of $G$ contains all limit ordinals.

Finally, consider $\gamma + 1$, where $\gamma$ is a limit ordinal.  The range of $G$ contains all limit ordinals, so there is an even ordinal $\alpha$ with $G(\alpha) = \gamma$.  Then $G(\alpha + 1) = \gamma + 1$.  So the range of $G$ contains all successors of limit ordinals as well.
\end{proof}

\begin{lemma}\label{lem-FGInv}
$F(G(\alpha)) = \alpha$ for all ordinals $\alpha$.
\end{lemma}

\begin{proof}
Proceed by transfinite induction on $\alpha$.  Clearly $F(G(0)) = 0$.  Now suppose that $F(G(\delta)) = \delta$ for all $\delta < \alpha$.

We have two successor cases, depending on whether $\alpha$ is even or odd.  First, suppose that $\alpha = \beta + 2d + 1$, where $d < \omega$ and $\beta$ is either $0$ or a limit.  Then
\begin{align*}
F(G(\alpha)) &= F(G(\beta + 2d + 1)) & &\\
&= F(G(\beta + 2d) + 1) & &\text{by the definition of $G$}\\
&= F(G(\beta + 2d)) + 1 & &\text{by the definition of $F$, as $G(\beta + 2d)$ is $0$ or a limit}\\
&=\beta + 2d + 1 & &\text{by the induction hypothesis}\\
&=\alpha. & &\\
\intertext{Second, suppose that $\alpha = \beta + 2d + 2$, where $d < \omega$ and $\beta$ is either $0$ or a limit.  Then}
F(G(\alpha)) &= F(G(\beta + 2d + 2)) & &\\
&= F(G(\beta + 2d + 1) + \omega) & &\text{by the definition of $G$}\\
&= \lim_{n < \omega}\Bigl[F(G(\beta + 2d + 1) + n) + 1\Bigr] & &\text{by the definition of $F$}\\
&= \lim_{n < \omega}\Bigl[F(G(\beta + 2d) + 1 + n) + 1\Bigr] & &\text{by the definition of $G$}\\
&= \lim_{n < \omega}\Bigl[F(G(\beta + 2d)) + 1 + 1\Bigr] & &\text{by the definition of $F$, as $G(\beta + 2d)$ is $0$ or a limit}\\
&= \lim_{n < \omega}\Bigl[\beta + 2d + 2 \Bigr] & &\text{by the induction hypothesis}\\
&= \beta + 2d + 2 & &\\
&= \alpha. & &
\end{align*}
Finally, suppose that $\alpha$ is a limit.  Then $G(\alpha)$ is a limit as well, and
\begin{align*}
F(G(\alpha)) &= \lim_{\beta < G(\alpha)}(F(\beta) + 1).
\end{align*}
If $\beta < G(\alpha)$, then $\beta < G(\delta)$ for some $\delta < \alpha$ because $G(\alpha) = \lim_{\delta < \alpha}G(\delta)$.  Then $F(\beta) + 1 \leq F(G(\delta)) + 1 = \delta + 1 < \alpha$ by the induction hypothesis and the fact that $\alpha$ is a limit.  Thus $F(\beta) + 1 < \alpha$ for all $\beta < G(\alpha)$, so $F(G(\alpha)) \leq \alpha$.  On the other hand, if $\beta < \alpha$, then $G(\beta) < G(\alpha)$, so $F(G(\alpha)) \geq F(G(\beta)) + 1 = \beta + 1 > \beta$ by the induction hypothesis.  Thus $F(G(\alpha)) \geq \alpha$.  So $F(G(\alpha)) = \alpha$.
\end{proof}

The next two lemmas show that $\sim_\alpha$ is $\Pi^0_{1 + F(\alpha)}$-definable when $\alpha \geq 1$ is a computable successor ordinal and is $\Sigma^0_{1+F(\alpha)}$-definable when $\alpha$ is a computable limit ordinal.

\begin{lemma}\label{lem-HypDefG}
Let $\alpha \geq 1$ be a computable ordinal.
\begin{itemize}
\item If $\alpha$ is even, then $\sim_{G(\alpha)}$ is $\Sigma^0_{1+\alpha}$-definable.
\item If $\alpha$ is odd, then $\sim_{G(\alpha)}$ is $\Pi^0_{1+\alpha}$-definable.
\end{itemize}
\end{lemma}

\begin{proof}
By effective transfinite recursion on $\alpha$, we uniformly produce a code for a formula $\varphi(u,v)$ with the following properties.
\begin{itemize}
\item If $\alpha$ is a even, then $\varphi(u,v)$ is $\Sigma^0_{1 + \alpha}$.
\item If $\alpha$ is a odd, then $\varphi(u,v)$ is $\Pi^0_{1 + \alpha}$.
\item For all closed terms $s$ and $t$, $s \sim_{G(\alpha)} t$ if and only if $\varphi(s,t)$.
\end{itemize}

For the base case $\alpha = 1$, we have that $1$ is odd, that $G(1) = 1$, and that $\sim_1$ is defined by the fixed $\Pi^0_2$ formula $\varphi(u, v) \equiv \forall x (u x \simeq v x)$.

Suppose that $\alpha = \beta + 1$ is a successor.  There are two cases, depending on whether $\beta$ even or odd.  The coding of computable ordinals allows us to distinguish between the cases.

First, suppose that $\beta$ is even.  In this case, $\alpha$ is odd and $G(\alpha) = G(\beta + 1) = G(\beta) + 1$.  By effective transfinite recursion, we can compute a (code for a) $\Sigma^0_{1 + \beta}$ formula $\psi(u, v)$ such that $s \sim_{G(\beta)} t \Biimp \psi(s,t)$ for all closed terms $s$ and $t$.  Let $\varphi(u, v)$ be the $\Pi^0_{1 + \beta + 1}$ formula $\varphi(u, v) \equiv \forall x \; \psi(u x, v x)$.  Then $s \sim_{G(\beta) + 1} t \Biimp \varphi(s,t)$ for all closed terms $s$ and $t$.  Thus $\varphi(u, v)$ is the desired formula because $\alpha = \beta + 1$ is odd and and $G(\alpha) = G(\beta) + 1$.

Second, suppose that $\beta$ is odd.  In this case, $\alpha$ is even and $G(\alpha) = G(\beta + 1) = G(\beta) + \omega$.  By effective transfinite recursion, we can compute a (code for a) $\Pi^0_{1 + \beta}$ formula $\psi(u, v)$ such that $s \sim_{G(\beta)} t \Biimp \psi(s,t)$ for all closed terms $s$ and $t$.  From the code for $\psi$, we can uniformly compute codes for $\Pi^0_{1 + \beta}$ formulas $\psi_n(u, v)$, where $\psi_n(u, v) \equiv \forall x_0, \dots, x_{n-1} \; \psi(u x_0 \cdots x_{n-1}, v x_0 \cdots x_{n-1})$ for each $n < \omega$.  Note that $s \sim_{G(\beta) + n} t \Biimp \psi_n(s,t)$ for all closed terms $s$ and $t$ and all $n < \omega$.  We can then compute a code for the $\Sigma^0_{1 + \beta + 1}$ formula $\varphi(u, v) \equiv \bigvee_{n \in \omega} \psi_n(u, v)$.  For closed terms $s$ and $t$, $\varphi(s,t)$ holds if and only if $\psi_n(s,t)$ holds for some $n < \omega$ if and only if $s \sim_{G(\beta) + n} t$ for some $n < \omega$ if and only if $s \sim_{G(\beta) + \omega} t$.  Thus $\varphi(u, v)$ is the desired formula because $\alpha = \beta + 1$ is even and $G(\alpha) = G(\beta) + \omega$.

Now suppose that $\alpha$ is a limit, and let $\beta_0 < \beta_1 < \cdots$ be a computable sequence of ordinals converging to $\alpha$.  We may assume that each $\beta_n$ is odd by adding to $\beta_n$ where necessary.  Note that $\alpha$ is even and that $G(\alpha) = \lim_{n < \omega}G(\beta_n)$ is a limit ordinal.  By effective transfinite recursion, for each $n < \omega$, we can uniformly compute (a code for) a $\Pi^0_{1 + \beta_n}$ formula $\psi_n(u, v)$ such that $s \sim_{G(\beta_n)} t \Biimp \psi_n(s,t)$ for all closed terms $s$ and $t$.  We can thus compute a code for the $\Sigma^0_{1 + \alpha}$ formula $\varphi(u, v) \equiv \bigvee_{n \in \omega} \psi_n(u, v)$.  For closed terms $s$ and $t$, $\varphi(s,t)$ holds if and only if $\psi_n(s,t)$ holds for some $n < \omega$ if and only if $s \sim_{G(\beta_n)} t$ for some $n < \omega$ if and only if $s \sim_{G(\alpha)} t$.  Thus $\varphi(u, v)$ is the desired formula.
\end{proof}

\begin{lemma}\label{lem-HypDefF}
Let $\alpha \geq 1$ be a computable ordinal.
\begin{itemize}
\item If $\alpha$ is a successor, then $\sim_\alpha$ is $\Pi^0_{1 + F(\alpha)}$-definable.
\item If $\alpha$ is a limit, then $\sim_\alpha$ is $\Sigma^0_{1 + F(\alpha)}$-definable.
\end{itemize}
\end{lemma}

\begin{proof}
Suppose that $\alpha \geq 1$ is a successor.  Write $\alpha$ as $\alpha = \gamma + 1 + d$, where $\gamma$ is either $0$ or a limit and $d < \omega$.  By Lemma~\ref{lem-RanG}, let $\beta$ be such that $G(\beta) = \gamma + 1$.  Then $\beta$ is odd, and ${\sim_{G(\beta)}} = {\sim_{\gamma + 1}}$ is $\Pi^0_{1+\beta}$-definable by Lemma~\ref{lem-HypDefG}.  Therefore ${\sim_\alpha} = {\sim_{\gamma + 1 + d}}$ is definable by a $\Pi^0_{1+\beta}$ formula equivalent to $\forall x_0, \dots, x_{d-1} \; (u x_0 \cdots x_{d-1} \sim_{\gamma + 1} v x_0 \cdots x_{d-1})$.  Furthermore, $F(\alpha) = F(\gamma + 1 + d) = F(\gamma + 1) = F(G(\beta)) = \beta$, where the last equality is by Lemma~\ref{lem-FGInv}.  So $\sim_\alpha$ is $\Pi^0_{1 + F(\alpha)}$-definable.

Suppose that $\alpha$ is a limit.  By Lemma~\ref{lem-RanG}, let $\beta$ be such that $G(\beta) = \alpha$.  Then $\beta$ is even, and ${\sim_{G(\beta)}} = {\sim_\alpha}$ is $\Sigma^0_{1+\beta}$-definable by Lemma~\ref{lem-HypDefG}.  We have that $F(\alpha) = F(G(\beta)) = \beta$ by Lemma~\ref{lem-FGInv}, so $\sim_\alpha$ is $\Sigma^0_{1 + F(\alpha)}$-definable.
\end{proof}

We now show that the $\sim_\alpha$ relations for computable $\alpha \geq 1$ are hard for the complexity classes indicated in Lemma~\ref{lem-HypDefF}.  The main tool for handling computable infinitary disjunctions is Lemma~\ref{lem-CombineOmega}.  Lemma~\ref{lem-Combine2} is a simpler version of Lemma~\ref{lem-CombineOmega} (and we think a more pleasing statement) that we prove first to illustrate the idea.

Using the recursion theorem, let $e_* \in \K_1$ be such that $e_* \cdot x = e_*$ for every $x$.  This $e_*$ remains fixed for the rest of this section.  Notice that it is hereditarily total.

For the rest of this section, $\la \cdot \ra \colon \omega^{<\omega} \imp \omega$ denotes a computable bijective encoding of finite strings, and $\la \cdot \ra_n \colon \omega^n \imp \omega$ denotes a uniformly computable sequence of bijective encodings of $n$-tuples for each $n > 0$.

\begin{lemma}\label{lem-Combine2}
There is a total computable function $f(a,b)$ such that for every hereditarily total $a$ and $b$ in $\K_1$ and every ordinal $\alpha$, $f(a,b) \sim_\alpha e_*$ if and only if either $a \sim_\alpha e_*$ or $b \sim_\alpha e_*$.
\end{lemma}

\begin{proof}
Let $g(i,x,y)$ be a total computable function such that for all $i$, $x$, $y$, $u$, and $v$,
\begin{align*}
\Phi_{g(i,x,y)}(\la u, v \ra_2) \simeq \Phi_i(xu, yv).
\end{align*}
By padding, we may define $g$ so that $e_* \notin \ran(g)$.  Using the recursion theorem, let $i$ be such that
\begin{align*}
\Phi_i(x, y) =
\begin{cases}
e_* & \text{if $x = e_*$ or $y = e_*$}\\
g(i, x, y) & \text{otherwise}.
\end{cases}
\end{align*}
Note that $\Phi_i$ is total because $g$ is total.  Also note that $\Phi_i(x,y) = e_*$ if and only if $x = e_*$ or $y = e_*$ because $e_* \notin \ran(g)$.  Let $f$ be the function computed by $\Phi_i$.  This $f$ is the desired function.

\begin{claim*}
Consider any $x$, $y$, $u$, and $v$ in $\K_1$.  If $xu\da$ and $yv\da$, then $f(x,y) \cdot \la u, v \ra_2 = f(xu, yv)$.
\end{claim*}

\begin{proof}[Proof of Claim]
If $x = e_*$ or $y = e_*$, then $xu = e_*$ or $yv = e_*$, so $f(x, y) \cdot \la u, v \ra_2 = e_* = f(xu, yv)$.  If neither $x$ nor $y$ is $e_*$, then
\begin{align*}
f(x,y) \cdot \la u, v \ra_2 = \Phi_i(x, y) \cdot \la u, v \ra_2 = g(i, x, y) \cdot \la u, v \ra_2 = \Phi_{g(i,x,y)}(\la u, v \ra_2) = \Phi_i(xu, yv)  = f(xu, yv).
\end{align*}
\end{proof}

If $x$ and $y$ are hereditarily total, then so is $f(x,y)$ because in this case, by repeated application of the Claim,
\begin{align*}
f(x,y) \cdot \la u_0, v_0 \ra_2 \cdots \la u_{n-1}, v_{n-1} \ra_2 = f(x u_0 \cdots u_{n-1}, y v_0 \cdots v_{n-1}),
\end{align*}
which is defined for every sequence $\la u_0, v_0 \ra_2, \dots, \la u_{n-1}, v_{n-1} \ra_2$.

We show by transfinite induction on $\alpha$ that if $a$ and $b$ are hereditarily total elements of $\K_1$, then $f(a, b) \sim_\alpha e_*$ if and only if $a \sim_\alpha e_*$ or $b \sim_\alpha e_*$.

For the base case, first suppose that $a \sim_0 e_*$.  Then $f(a,b) = e_*$, so $f(a,b) \sim_0 e_*$.  Similarly, if $b \sim_0 e_*$, then $f(a, b) \sim_0 e_*$.  Conversely, suppose that $a \not\sim_0 e_*$ and that $b \not\sim_0 e_*$.  Then $a \neq e_0$ and $b \neq e_0$, in which case $f(a, b) \neq e_*$.  Thus $f(a, b) \not\sim_0 e_*$.

For the successor case, first suppose that $a \sim_{\alpha + 1} e_*$.  Then $au \sim_\alpha e_* \cdot u = e_*$ for every $u$.  For every $\la u, v \ra_2$, both $au$ and $bv$ are hereditarily total because $a$ and $b$ are.  Thus by the induction hypothesis and the Claim, $f(a,b) \cdot \la u, v \ra_2 = f(au, bv) \sim_\alpha e_*$ for every $\la u, v \ra_2$.  Therefore $f(a,b) \sim_{\alpha + 1} e_*$.  Similarly, if $b \sim_{\alpha + 1} e_*$, then $f(a,b) \sim_{\alpha + 1} e_*$.

Conversely, suppose that $a \not\sim_{\alpha + 1} e_*$ and $b \not\sim_{\alpha + 1} e_*$.  Then there are $u_0$ and $v_0$ such that $a u_0 \not\sim_\alpha e_*$ and $b v_0 \not\sim_\alpha e_*$.  Both $a u_0$ and $b v_0$ are hereditarily total because $a$ and $b$ are.  Thus $f(a, b) \cdot \la u_0, v_0 \ra_2 = f(a u_0, b v_0) \not\sim_\alpha e_*$ by the induction hypothesis and the Claim, so $f(a, b) \not\sim_{\alpha + 1} e_*$.

For the limit case, suppose that $\alpha$ is a limit ordinal.  Then $f(a, b) \sim_\alpha e_*$ if and only if there is a $\beta < \alpha$ such that $f(a, b) \sim_\beta e_*$.  By the induction hypothesis, this holds if and only if there is a $\beta < \alpha$ such that either $a \sim_\beta e_*$ or $b \sim_\beta e_*$, which in turn holds if and only if $a \sim_\alpha e_*$ or $b \sim_\alpha e_*$.
\end{proof}

\begin{lemma}\label{lem-CombineOmega}
There is a computable procedure which, given an index for a computable sequence $a_0, a_1, \dots$ of hereditarily total indices in $\K_1$, produces a hereditarily total index $w \in \K_1$ such that the following hold for every ordinal $\alpha$.
\begin{itemize}
\item If $a_n \sim_\alpha e_*$ for some $n$, then $w \sim_{\alpha + n + 1} e_*$.
\item If $a_n \not\sim_\alpha e_*$ for every $n$, then $w \not\sim_\alpha e_*$.
\end{itemize}
Formally, there is a total computable function $f(e)$ such that if $\Phi_e(n)\da = a_n$ is hereditarily total for every $n$, then $f(e)$ is the desired $w$.
\end{lemma}

\begin{proof}
Let $g(i,e,\la x_0, \dots, x_{n-1} \ra)$ be a total computable function such that
\begin{align*}
\Phi_{g(i,e,\la x_0, \dots, x_{n-1} \ra)}(\la u_0, \dots, u_{n-1} \ra_n) &\simeq \Phi_i(e, \la x_0u_0, \dots, x_{n-1}u_{n-1}, \Phi_e(n)\ra) & &\text{when $n > 0$}\\
\Phi_{g(i,e, \epsilon)}(u) &\simeq \Phi_i(e, \la \Phi_e(0) \ra) & &\text{when $n = 0$}.
\end{align*}
By padding, we may ensure that $e_* \notin \ran(g)$.  Using the recursion theorem, let $i$ be such that
\begin{align*}
\Phi_i(e, \la x_0, \dots, x_{n-1} \ra) =
\begin{cases}
e_* & \text{if $x_\ell = e_*$ for some $\ell < n$}\\
g(i, e, \la x_0, \dots, x_{n-1} \ra) & \text{otherwise}.
\end{cases}
\end{align*}
Note that $\Phi_i$ is total because $g$ is total.  Also note that $\Phi_i(e, \la x_0, \dots, x_{n-1} \ra) = e_*$ if and only if $x_\ell = e_*$ for some $\ell < n$ because $e_* \notin \ran(g)$.  Let $h$ be the function computed by $\Phi_i$.

\begin{claim}\label{clm-app}
Consider any $n > 0$, $\la x_0, \dots, x_{n-1} \ra$, $\la u_0, \dots, u_{n-1} \ra_n$, and $e$.  If $(\forall \ell < n)(x_\ell u_\ell\da)$ and $\Phi_e(n)\da$, then
\begin{align*}
h(e, \la x_0, \dots, x_{n-1} \ra) \cdot \la u_0, \dots, u_{n-1} \ra_n = h(e, \la x_0u_0, \dots, x_{n-1}u_{n-1}, \Phi_e(n) \ra).
\end{align*}
Furthermore, if $\Phi_e(0)\da$, then for every $u$,
\begin{align*}
h(e, \epsilon) \cdot u = h(e, \la \Phi_e(0) \ra).
\end{align*}
\end{claim}

\begin{proof}[Proof of Claim]
For $n = 0$,
\begin{align*}
h(e, \epsilon) \cdot u = \Phi_i(e, \epsilon) \cdot u = g(i, e, \epsilon) \cdot u = \Phi_{g(i, e, \epsilon)}(u) = \Phi_i(e, \la \Phi_e(0) \ra) = h(e, \la \Phi_e(0) \ra).
\end{align*}

For $n > 0$, first suppose that $x_\ell = e_*$ for some $\ell < n$.  Then also $x_\ell u_\ell = e_*$, so
\begin{align*}
h(e, \la x_0, \dots, x_{n-1} \ra) \cdot \la u_0, \dots, u_{n-1} \ra_n &= \Phi_i(e, \la x_0, \dots, x_{n-1} \ra) \cdot \la u_0, \dots, u_{n-1} \ra_n \\
&= e_* \cdot \la u_0, \dots, u_{n-1} \ra_n \\
&= e_* \\
&= \Phi_i(e, \la x_0u_0, \dots, x_{n-1}u_{n-1}, \Phi_e(n) \ra) \\
&= h(e, \la x_0u_0, \dots, x_{n-1}u_{n-1}, \Phi_e(n) \ra).
\end{align*}

Now suppose that $x_\ell \neq e_*$ for all $\ell < n$.  Then
\begin{align*}
h(e, \la x_0, \dots, x_{n-1} \ra) \cdot \la u_0, \dots, u_{n-1} \ra_n &= \Phi_i(e, \la x_0, \dots, x_{n-1} \ra) \cdot \la u_0, \dots, u_{n-1} \ra_n \\
&= g(i, e, \la x_0, \dots, x_{n-1} \ra) \cdot \la u_0, \dots, u_{n-1} \ra_n \\
&= \Phi_{g(i, e, \la x_0, \dots, x_{n-1} \ra)}(\la u_0, \dots, u_{n-1} \ra_n)\\
&= \Phi_i(e, \la x_0u_0, \dots, x_{n-1}u_{n-1}, \Phi_e(n) \ra) \\
&= h(e, \la x_0u_0, \dots, x_{n-1}u_{n-1}, \Phi_e(n) \ra).
\end{align*}
\end{proof}

Using Claim~\ref{clm-app}, we show that if $x_\ell$ is hereditarily total for all $\ell < n$ and $\Phi_e(m)$ is defined and hereditarily total for all $m \geq n$, then $h(e, \la x_0, \dots, x_{n-1} \ra)$ is hereditarily total.  To do this, we show by induction on $r$ that for every $v_0, \dots, v_{r-1}$, $h(e, \la x_0, \dots, x_{n-1} \ra) \cdot v_0 \cdots v_{r-1}$ is defined and of the form $h(e, \la y_0, \dots, y_{n+r-1} \ra)$, where $y_\ell$ hereditarily total for each $\ell < n+r$.  The base case $r=0$ is trivial.  Consider $r+1$ and $h(e, \la x_0, \dots, x_{n-1} \ra)\cdot v_0 \cdots v_{r-1}v_r$.  By the induction hypothesis,
\begin{align*}
h(e, \la x_0, \dots, x_{n-1} \ra)\cdot v_0 \cdots v_{r-1} = h(e, \la y_0, \dots, y_{n+r-1} \ra),
\end{align*}
where $y_\ell$ is hereditarily total for each $\ell < n+r$.  Decode $v_r$ via $\la \cdot \ra_{n+r}$ as $v_r = \la u_0, \dots, u_{n+r-1} \ra_{n+r}$.  Then
\begin{align*}
h(e, \la x_0, \dots, x_{n-1} \ra)\cdot v_0 \cdots v_{r-1}v_r &= h(e, \la y_0, \dots, y_{n+r-1} \ra) \cdot \la u_0, \dots, u_{n+r-1} \ra_{n+r}\\
&= h(e, \la y_0u_0, \dots, y_{n+r-1}u_{n+r-1}, \Phi_e(n+r) \ra)
\end{align*}
by Claim~\ref{clm-app}, where $y_\ell u_\ell$ is hereditarily total for each $\ell < n+r$ because $y_\ell$ is hereditarily total for each $\ell < n+r$, and $\Phi_e(n+r)$ is defined and hereditarily total by assumption.

\begin{claim}\label{clm-PresEquiv}
Let $\la x_0, \dots, x_{n-1} \ra$, $e$, and $\alpha$ be such that 
\begin{itemize}
\item $(\forall \ell < n)(\text{$x_\ell$ is hereditarily total})$,
\item $(\forall m \geq n)(\text{$\Phi_e(m)\da$ is hereditarily total})$, and 
\item $(\exists \ell < n)(x_\ell \sim_\alpha e_*)$.
\end{itemize}
Then $h(e, \la x_0, \dots, x_{n-1} \ra) \sim_\alpha e_*$.
\end{claim}

\begin{proof}[Proof of Claim]
Proceed by transfinite induction on $\alpha$.  Notice that if $(\exists \ell < n)(x_\ell \sim_\alpha e_*)$, then $n > 0$.  For the base case, consider $\la x_0, \dots, x_{n-1} \ra$, and suppose that $x_\ell \sim_0 e_*$ for some $\ell < n$.  Then $x_\ell = e_*$, so $h(e, \la x_0, \dots, x_{n-1} \ra) = \Phi_i(e, \la x_0, \dots, x_{n-1} \ra) = e_*$.  Thus $h(e, \la x_0, \dots, x_{n-1} \ra) \sim_0 e_*$.

For the successor case, consider $\la x_0, \dots, x_{n-1} \ra$, and suppose that $x_\ell \sim_{\alpha + 1} e_*$ for some $\ell < n$.  For every $u = \la u_0, \dots, u_{n-1} \ra_n$, we have that $x_\ell u_\ell \sim_\alpha e_*$.  Then by Claim~\ref{clm-app},
\begin{align*}
h(e, \la x_0, \dots, x_{n-1} \ra) \cdot u &= h(e, \la x_0, \dots, x_{n-1} \ra) \cdot \la u_0, \dots, u_{n-1} \ra_n\\
&= h(e, \la x_0u_0, \dots, x_{n-1}u_{n-1}, \Phi_e(n) \ra).
\end{align*}
Thus by the induction hypothesis, we have that $h(e, \la x_0u_0, \dots, x_{n-1}u_{n-1}, \Phi_e(n) \ra) \sim_\alpha e_*$.  Therefore $h(e, \la x_0, \dots, x_{n-1} \ra) \cdot u \sim_\alpha e_*$ for every $u$, so $h(e, \la x_0, \dots, x_{n-1} \ra) \sim_{\alpha + 1} e_*$.

For the limit case, suppose that $\alpha$ is a limit, consider $\la x_0, \dots, x_{n-1} \ra$, and suppose that $x_\ell \sim_\alpha e_*$ for some $\ell < n$.  Then $x_\ell \sim_\beta e_*$ for some $\beta < \alpha$.  By the induction hypothesis, $h(e, \la x_0, \dots, x_{n-1} \ra) \sim_\beta e_*$ as well, so $h(e, \la x_0, \dots, x_{n-1} \ra) \sim_\alpha e_*$.
\end{proof}

\begin{claim}\label{clm-PresNotEquiv}
Let $\la x_0, \dots, x_{n-1} \ra$, $e$, and $\alpha$ be such that 
\begin{itemize}
\item $(\forall \ell < n)(\text{$x_\ell$ is hereditarily total and $x_\ell \not\sim_\alpha e_*$})$ and
\item $(\forall m \geq n)(\text{$\Phi_e(m)\da$ is hereditarily total and $\Phi_e(m) \not\sim_\alpha e_*$})$.
\end{itemize}
Then $h(e, \la x_0, \dots, x_{n-1} \ra) \not\sim_\alpha e_*$.
\end{claim}

\begin{proof}[Proof of Claim]
Proceed by transfinite induction on $\alpha$.  For the base case, consider $\la x_0, \dots, x_{n-1} \ra$.  Suppose that $(\forall \ell < n)(x_\ell \not\sim_0 e_*)$ and that $(\forall m \geq n)(\Phi_e(m) \not\sim_0 e_*)$.  Then $(\forall \ell < n)(x_\ell \neq e_*)$, so
\begin{align*}
h(e, \la x_0, \dots, x_{n-1} \ra) = \Phi_i(e, \la x_0, \dots, x_{n-1} \ra) = g(i, e, \la x_0, \dots, x_{n-1} \ra) \neq e_*
\end{align*}
because $e_* \notin \ran(g)$.  Thus $h(e, \la x_0, \dots, x_{n-1} \ra) \not\sim_0 e_*$.

For the successor case, consider $\la x_0, \dots, x_{n-1} \ra$.  Suppose that $(\forall \ell < n)(x_\ell \not\sim_{\alpha+1} e_*)$ and that $(\forall m \geq n)(\Phi_e(m) \not\sim_{\alpha+1} e_*)$.  If $n = 0$, then $\la x_0, \dots, x_{n-1} \ra = \epsilon$.  In this case, for any $u$ we have that $h(e, \epsilon) \cdot u = h(e, \la \Phi_e(0) \ra) \not\sim_\alpha e_*$ by Claim~\ref{clm-app}, the assumption that $\Phi_e(0) \not\sim_{\alpha+1} e_*$ and hence $\Phi_e(0) \not\sim_\alpha e_*$, and the induction hypothesis.  Therefore $h(e, \epsilon) \not\sim_{\alpha + 1} e_*$.  Now assume that $n > 0$.  For each $\ell < n$, let $u_\ell$ be such that $x_\ell u_\ell \not\sim_\alpha e_*$.  Then by Claim~\ref{clm-app},
\begin{align*}
h(e, \la x_0, \dots, x_{n-1} \ra) \cdot \la u_0, \dots, u_{n-1} \ra_n = h(e, \la x_0u_0, \dots, x_{n-1}u_{n-1}, \Phi_e(n) \ra),
\end{align*}
and the sequence $\la y_0, \dots, y_{n-1}, y_n \ra = \la x_0u_0, \dots, x_{n-1}u_{n-1}, \Phi_e(n) \ra$ satisfies $(\forall \ell < n+1)(y_\ell \not\sim_\alpha e_*)$ and $(\forall m \geq n+1)(\Phi_e(m) \not\sim_\alpha e_*)$.  Thus $h(e, \la x_0u_0, \dots, x_{n-1}u_{n-1}, \Phi_e(n) \ra) \not\sim_\alpha e_*$ by the induction hypothesis.  Therefore $h(e, \la x_0, \dots, x_{n-1} \ra) \cdot \la u_0, \dots, u_{n-1} \ra_n \not\sim_\alpha e_*$, so $h(e, \la x_0, \dots, x_{n-1} \ra) \not\sim_{\alpha + 1} e_*$.

For the limit case, suppose that $\alpha$ is a limit, and consider $\la x_0, \dots, x_{n-1} \ra$.  Suppose that $(\forall \ell < n)(x_\ell \not\sim_\alpha e_*)$ and that $(\forall m \geq n)(\Phi_e(m) \not\sim_\alpha e_*)$.  Then for every $\beta < \alpha$, it is also the case that $(\forall \ell < n)(x_\ell \not\sim_\beta e_*)$ and that $(\forall m \geq n)(\Phi_e(m) \not\sim_\beta e_*)$.  Therefore $h(e, \la x_0, \dots, x_{n-1} \ra) \not\sim_\beta e_*$ for all $\beta < \alpha$ by the induction hypothesis.  Thus $h(e, \la x_0, \dots, x_{n-1} \ra) \not\sim_\alpha e_*$.
\end{proof}

Now define $f$ by $f(e) = h(e, \epsilon)$.  This is our desired $f$.  Suppose that $\Phi_e(n)\da = a_n$ is hereditarily total for every $n$.  Let $w = f(e)$.  Then $w$ is hereditarily total by the discussion following Claim~\ref{clm-app}.  To finish the proof of the lemma, we must show the following for every ordinal $\alpha$.
\begin{itemize}
\item If $a_n \sim_\alpha e_*$ for some $n$, then $w \sim_{\alpha + n + 1} e_*$.
\item If $a_n \not\sim_\alpha e_*$ for every $n$, then $w \not\sim_\alpha e_*$.
\end{itemize}

First suppose that $a_n \sim_\alpha e_*$ for some $n$ and $\alpha$.  Consider any $v_0, \dots, v_n$, and for each $1 \leq \ell \leq n$, decode $v_\ell$ via $\la \cdot \ra_\ell$ as $v_\ell = \la u^\ell_0, \dots u^\ell_{\ell-1} \ra_\ell$.  By repeated applications of Claim~\ref{clm-app},
\begin{align*}
w v_0v_1v_2v_3\cdots v_n &= f(e) \cdot v_0v_1v_2v_3 \cdots v_n \\
&= h(e, \epsilon) \cdot v_0v_1v_2v_3 \cdots v_n \\
&= h(e, \epsilon) \cdot v_0 \cdot \la u^1_0 \ra_1 \cdot \la u^2_0, u^2_1 \ra_2 \cdot \la u^3_0, u^3_1, u^3_2 \ra_3 \cdots \la u^n_0, \dots, u^n_{n-1} \ra_n \\
&= h(e, \la a_0 \ra) \cdot \la u^1_0 \ra_1 \cdot \la u^2_0, u^2_1 \ra_2 \cdot \la u^3_0, u^3_1, u^3_2 \ra_3 \cdots \la u^n_0, \dots, u^n_{n-1} \ra_n \\
&= h(e, \la a_0u^1_0,\; a_1 \ra) \cdot \la u^2_0, u^2_1 \ra_2 \cdot \la u^3_0, u^3_1, u^3_2 \ra_3 \cdots \la u^n_0, \dots, u^n_{n-1} \ra_n \\
&= h(e, \la a_0u^1_0u^2_0,\; a_1u^2_1,\; a_2 \ra) \cdot \la u^3_0, u^3_1, u^3_2 \ra_3 \cdots \la u^n_0, \dots, u^n_{n-1} \ra_n\\
&\;\;\vdots \\
&= h(e, \la a_0u^1_0u^2_0u^3_0 \cdots u^n_0,\; a_1u^2_1u^3_1 \cdots u^n_1,\; a_2u^3_2 \cdots u^n_2, \;\dots,\; a_{n-1}u^{n}_{n-1},\; a_n \ra).
\end{align*}
Let $x_\ell$ denote $a_\ell u^{\ell+1}_\ell u^{\ell+2}_\ell \cdots u^n_\ell$ for each $\ell \leq n$ (with $x_n = a_n$), so that we may write
\begin{align*}
w v_0v_1v_2v_3\cdots v_n &= h(e, \la x_0, \dots, x_n \ra).
\end{align*}
Each $x_\ell$ is hereditarily total because each $a_\ell$ is hereditarily total.  We assumed that $x_n = a_n \sim_\alpha e_*$, so $h(e, \la x_0, \dots, x_n \ra) \sim_\alpha e_*$ by Claim~\ref{clm-PresEquiv}.  The initial $v_0, \dots, v_n$ were arbitrary, so we have shown that $w v_0 \cdots v_n \sim_\alpha e_*$ for every $v_0, \dots, v_n$.  Therefore $w \sim_{\alpha + n + 1} e_*$, as desired.  

Finally, suppose that $a_n \not\sim_\alpha e_*$ for every $n$.  Then $w = f(e) = h(e, \epsilon) \not\sim_\alpha e_*$ by Claim~\ref{clm-PresNotEquiv}.  This completes the proof of the lemma.
\end{proof}

\begin{lemma}\label{lem-HardnessHelper}
Let $\varphi(z)$ be a computable $\Sigma^0_\alpha$ formula for a computable $\alpha \geq 1$.  There is a computable procedure which, for every $z$, produces a hereditarily total index $p_z\in \K_1$ satisfying the following conditions.
\begin{itemize}
\item If $\alpha$ is even and $\varphi(z)$ holds, then $p_z \sim_{G(\alpha)} e_*$.
\item If $\alpha$ is even and $\neg\varphi(z)$ holds, then $p_z \napprox e_*$.
\item If $\alpha$ is odd and $\neg\varphi(z)$ holds, then $p_z\sim_{G(\alpha)} e_*$.
\item If $\alpha$ is odd and $\varphi(z)$ holds, then $p_z \napprox e_*$.
\end{itemize}
\end{lemma}

\begin{proof}
We produce $p_z$ by effective transfinite recursion on $\alpha \geq 1$.  The $\Sigma^0_\alpha$ formula $\varphi(z)$ is indexed as c.e.\ disjunction
\begin{align*}
\varphi(z) \equiv \bigvee_{\ell \in \omega} \exists n \; \psi_\ell(z, n),
\end{align*}
where each $\psi_\ell(z,n)$ is a computable $\Pi^0_{\beta_\ell}$ formula for some $\beta_\ell < \alpha$.

For the base case $\alpha = 1$, for every $z$ we want a hereditarily total $p_z$ such that $p_z \sim_1 e_*$ if $\neg\varphi(z)$ holds and $p_z \napprox e_*$ if $\varphi(z)$ holds.  In this case, $\varphi(z)$ is a computable $\Sigma^0_1$ formula, so each $\psi_\ell$ is given by an index $e_\ell$ where $\Phi_{e_\ell}(z,n)$ computes the characteristic function of $\psi_\ell(z,n)$.  Using the recursion theorem, fix a $c_* \neq e_*$ such that $c_* x = c_*$ for every $x$, and notice that $c_* \napprox e_*$.  Given $z$, compute an index $p_z$ so that for every coded pair $\la \ell, n \ra_2$
\begin{align*}
p_z \cdot \la \ell, n \ra_2 = 
\begin{cases}
c_* & \text{if $\Phi_{e_\ell}(z, n) = 1$}\\
e_* & \text{if $\Phi_{e_\ell}(z, n) = 0$}.
\end{cases}
\end{align*}
Note that $p_z$ is hereditarily total because it always outputs either $c_*$ or $e_*$, both of which are hereditarily total.  If $\neg\varphi(z)$ holds, then $\neg\psi_\ell(z, n)$ holds for every $\ell$ and $n$, so $p_z \cdot \la \ell, n \ra_2 = e_*$ for every $\la \ell, n \ra_2$.  Thus $p_z \sim_1 e_*$.  If $\varphi(z)$ holds, then $\psi_\ell(z, n)$ holds for some $\ell$ and $n$, in which case $p_z \cdot \la \ell, n \ra_2 = c_*$.  Therefore $p_z \napprox e_*$ because $c_* \napprox e_*$.

Now suppose that $\alpha \geq 1$ is even.  For each $\ell$, we may assume that $\beta_\ell$ is odd by adding $1$ to $\beta_\ell$ as necessary, and we can uniformly effectively compute an index for a computable $\Sigma^0_{\beta_\ell}$ formula equivalent to $\neg\psi_\ell(z, n)$.  By effective transfinite recursion, for each $z, \ell, n \in \omega$, we can uniformly compute a hereditarily total index $q^z_{\ell, n} \in \K_1$ such that 
\begin{itemize}
\item if $\psi_\ell(z,n)$ holds, then $q^z_{\ell, n} \sim_{G(\beta_\ell)} e_*$; and
\item if $\neg\psi_\ell(z,n)$ holds, then $q^z_{\ell, n} \napprox e_*$.
\end{itemize}
Given $z$, let $p_z$ be the $w$ provided by Lemma~\ref{lem-CombineOmega} for the computable sequence $(q^z_{\ell, n} : \ell, n \in \omega)$.  Then $p_z$ is hereditarily total.  If $\varphi(z)$ holds, then $\psi_\ell(z, n)$ holds for some $\ell$ and $n$.  In this case, $q^z_{\ell, n} \sim_{G(\beta_\ell)} e_*$, so $p_z \sim_{G(\beta_\ell) + \omega} e_*$.  If $\alpha = \beta + 2d + 2$ for $\beta = 0$ or $\beta$ a limit, then $\beta_\ell \leq \beta + 2d + 1$, so $G(\beta_\ell) + \omega \leq G(\beta + 2d + 1) + \omega = G(\alpha)$.  Therefore $p_z \sim_{G(\alpha)} e_*$, as desired.  If $\alpha$ is a limit, then $\beta_\ell + 1 < \alpha$, so $G(\beta_\ell + 1) = G(\beta_\ell) + \omega \leq G(\alpha)$.  So again $p_z \sim_{G(\alpha)} e_*$, as desired.  If $\neg\varphi(z)$ holds, then $\neg\psi_\ell(z,n)$ holds for every $\ell$ and $n$.  In this case, $q^z_{\ell, n} \napprox e_*$ for every $\ell$ and $n$, so $p_z \napprox e_*$ as desired.

Finally, suppose that $\alpha \geq 1$ is odd and that $\alpha = \beta + 2d +1$ for either $\beta = 0$ or $\beta$ a limit.  For each $\ell$, we may assume that $\beta_\ell$ is even by adding $1$ to $\beta_\ell$ as necessary, and again we can uniformly effectively compute an index for a computable $\Sigma^0_{\beta_\ell}$ formula equivalent to $\neg\psi_\ell(z, n)$.  By effective transfinite recursion, for each $z, \ell, n \in \omega$, we can uniformly compute a hereditarily total index $q^z_{\ell, n}$ such that 
\begin{itemize}
\item if $\neg\psi_\ell(z,n)$ holds, then $q^z_{\ell, n} \sim_{G(\beta_\ell)} e_*$; and
\item if $\psi_\ell(z,n)$ holds, then $q^z_{\ell, n} \napprox e_*$.
\end{itemize}
Define $p_z$ so that $p_z \cdot \la \ell, n \ra_2 = q^z_{\ell, n}$ for every coded pair $\la \ell, n \ra_2$.  Then $p_z$ is hereditarily total because every $q^z_{\ell, n}$ is hereditarily total.  If $\neg\varphi(z)$ holds, then $\neg\psi_\ell(z,n)$ holds for every $\ell$ and $n$.  In this case, $p_z \cdot \la \ell, n \ra_2 = q^z_{\ell, n} \sim_{G(\beta_\ell)} e_*$ for every $\la \ell, n \ra_2$.  For each $\ell$, $\beta_\ell \leq \beta + 2d$, so $G(\beta_\ell) \leq G(\beta+2d)$.  Thus $p_z \cdot \la \ell, n \ra_2 \sim_{G(\beta + 2d)} e_*$ for all $\la \ell, n \ra_2$.  Thus $p_z \sim_{G(\beta + 2d) + 1} e_*$.  We have that $G(\alpha) = G(\beta + 2d + 1) = G(\beta + 2d) + 1$, so $p_z \sim_{G(\alpha)} e_*$, as desired.  On the other hand, if $\varphi(z)$ holds, then $\psi_\ell(z,n)$ holds for some $\ell$ and $n$.  In this case, $p_z \cdot \la \ell, n \ra_2 = q^z_{\ell, n} \napprox e_*$, so $p_z \napprox e_*$, as desired.
\end{proof}

\begin{lemma}\label{lem-HardnessG}
Let $\alpha \geq 1$ be a computable ordinal.
\begin{itemize}
\item If $\alpha$ is even, then $\sim_{G(\alpha)}$ is $\Sigma^0_\alpha$-hard.
\item If $\alpha$ is odd, then $\sim_{G(\alpha)}$ is $\Pi^0_\alpha$-hard.
\end{itemize}
\end{lemma}

\begin{proof}
Let $\alpha \geq 1$ be a computable ordinal.

Suppose that $\alpha$ is even and that $\varphi(z)$ is a computable $\Sigma^0_\alpha$ formula.  Let $z \mapsto p_z$ be the computable map of Lemma~\ref{lem-HardnessHelper} for $\varphi(z)$.  Then the map $z \mapsto (p_z, e_*)$ is a many-one reduction witnessing that $\{z : \varphi(z)\} \leqm {\sim_{G(\alpha)}}$.  Thus $\sim_{G(\alpha)}$ is $\Sigma^0_\alpha$-hard.

Suppose that $\alpha$ is odd and that $\psi(z)$ is a computable $\Pi^0_\alpha$ formula.  Let $\varphi(z)$ be a computable $\Sigma^0_\alpha$ formula equivalent to $\neg\psi(z)$, and let $z \mapsto p_z$ be the computable map of Lemma~\ref{lem-HardnessHelper} for $\varphi(z)$.  Then the map $z \mapsto (p_z, e_*)$ is a many-one reduction witnessing that $\{z : \psi(z)\} \leqm {\sim_{G(\alpha)}}$.  Thus $\sim_{G(\alpha)}$ is $\Pi^0_\alpha$-hard.
\end{proof}

\begin{lemma}\label{lem-HardnessF}
Let $\alpha \geq 1$ be a computable ordinal.
\begin{itemize}
\item If $\alpha$ is a successor, then $\sim_\alpha$ is $\Pi^0_{1 + F(\alpha)}$-hard.
\item If $\alpha$ is a limit, then $\sim_\alpha$ is $\Sigma^0_{1 + F(\alpha)}$-hard.
\end{itemize}
\end{lemma}

\begin{proof}
For the ordinals below $\omega^2$:
\begin{align*}
1 + F(\omega k) &= 2k+1 \text{ for every $k > 0$} \\
1 + F(\omega k + n) &= 2k+2 \text{ for every $k$ and $n > 0$}.
\end{align*}
Thus the desired result is given by Theorem~\ref{thm:ArithComplexity}.

Suppose that $\alpha \geq \omega^2$ is a computable ordinal.  Then $F(\alpha) \geq \omega$, so $1 + F(\alpha) = F(\alpha)$.  Thus it suffices to show that $\sim_\alpha$ is $\Pi^0_{F(\alpha)}$-hard when $\alpha$ is a successor and that $\sim_\alpha$ is $\Sigma^0_{F(\alpha)}$-hard when $\alpha$ is a limit.

First suppose that $\alpha$ is a successor, and write $\alpha$ as $\alpha = \gamma + 1 + d$, where $\gamma$ is a limit and $d < \omega$.  By Lemma~\ref{lem-RanG}, let $\beta$ be such that $G(\beta) = \gamma + 1$.  Then $\beta$ is odd, and ${\sim_{G(\beta)}} = {\sim_{\gamma + 1}}$ is $\Pi^0_\beta$-hard by Lemma~\ref{lem-HardnessG}.  Furthermore, $F(\alpha) = F(\gamma + 1 + d) = F(\gamma + 1) = F(G(\beta)) = \beta$, where the last equality is by Lemma~\ref{lem-FGInv}.  Thus $\sim_{\gamma + 1}$ is $\Pi^0_{F(\alpha)}$-hard.  We have that ${\sim_\alpha} = {\sim_{\gamma + 1 + d}} \equivm {\sim_{\gamma + 1}}$ by repeated applications of Lemma~\ref{lem-ReduceBy1}, so $\sim_\alpha$ is $\Pi^0_{F(\alpha)}$-hard.

Now suppose that $\alpha$ is a limit.  By Lemma~\ref{lem-RanG}, let $\beta$ be such that $G(\beta) = \alpha$.  Then $\beta$ is even, and ${\sim_{G(\beta)}} = {\sim_\alpha}$ is $\Sigma^0_{\beta}$-hard by Lemma~\ref{lem-HardnessG}.  We have that $F(\alpha) = F(G(\beta)) = \beta$ by Lemma~\ref{lem-FGInv}, so $\sim_\alpha$ is $\Sigma^0_{F(\alpha)}$-hard.  This completes the proof.
\end{proof}

\begin{theorem}\label{thm-HypComplexity}
Let $\alpha \geq 1$ be a computable ordinal.
\begin{itemize}
\item If $\alpha$ is a successor, then $\sim_\alpha$ is $\Pi^0_{1 + F(\alpha)}$-complete.
\item If $\alpha$ is a limit, then $\sim_\alpha$ is $\Sigma^0_{1 + F(\alpha)}$-complete.
\end{itemize}
\end{theorem}

\begin{proof}
By Lemma~\ref{lem-HypDefF} and Lemma~\ref{lem-HardnessF}.
\end{proof}

\section{Embeddings and sub-pca's} \label{sec:embeddings}

In order to clarify the relation between the relativized pca's $\K_1^X$
and $\K_2$ we consider the following notion of embedding between pca's.

\begin{definition}\label{def-embedding}
For pca's $\A$ and $\B$, we call $\A$ a \emph{sub-pca} of $\B$
if $\A \subseteq \B$ and the application of $\A$ is the restriction 
of the application of $\B$ to $\A$.
We say that an injective mapping $F \colon \A \imp \B$ is an 
\emph{embedding}\footnote{
It should be noted that this notion of embedding differs from 
the existence of an applicative morphism that is used in 
van Oosten~\cite{vanOosten2006}.
Van Oosten generalizes the notion of relativized computation 
to pca's by constructing a relativized pca $\A[f]$, and shows 
that there is an applicative morphism of $\A$ into $\A[f]$, 
which roughly means that computations of $\A$ can be simulated 
in $\A[f]$.
In particular there is an applicative morphism of 
$\K_1$ into $\K_1[X] = \K_1^X$ for every $X$. 
This means that there is a code $e$ such that for all $n,m,k\in\omega$, 
$\Phi_n(m)\da = k$ if and only if $\Phi_e^X(n,m)\da = k$. 
However, Proposition~\ref{prop-EmbedK1Rel} does not follow from this.
Also, $\A$ is not a sub-pca of $\A[f]$, as the universe is the 
same in both cases, but application is different.}
if, for all $a, b, c \in \A$,
\begin{align*}
(a \cdot_\A b) \da = c \quad&\Longrightarrow\quad (F(a) \cdot_\B F(b)) \da = F(c) \\
(a \cdot_\A b) \ua \quad&\Longrightarrow\quad (F(a) \cdot_\B F(b)) \ua.
\end{align*}
\end{definition}
Some authors additionally require that if $\A$ is to be a sub-pca of $\B$, then $\A$ must contain elements $K$ and $S$ satisfying Definition~\ref{def-pca} for both $\A$ and $\B$.  Likewise, some authors require that an embedding of one pca into another must also preserve some choice of $K$ and $S$.  These extra requirements are not necessary for our purposes.

If $X$ and $Y$ are different subsets of $\omega$, then $\K_1^X$ is not a sub-pca of $\K_1^Y$, as although the underlying set is $\omega$ for both pca's, the application operations are different.  Similarly, we view plain Turing machines and oracle Turing machines as formally different mathematical objects 
with different encodings as natural numbers.  
Thus $\K_1$ is not a sub-pca of $\K_1^X$ for any $X \subseteq \omega$ because again the application operation is different in both pca's, even when, for example, $X = \emptyset$.  However, $\K_1^X$ embeds into $\K_1^Y$ via a computable embedding whenever $X \leqT Y$.  In fact, $\K_1^X$ embeds into $\K_1^Y$ if and only if $X \leqT Y$
(cf.\ Theorem~\ref{thm:iff} below).
Likewise, $\K_1$ embeds into $\K_1^X$ via a computable embedding for every $X$.

\begin{proposition} \label{prop-EmbedK1Rel}
For every $X, Y \subseteq \omega$, if $X \leqT Y$, then $\K_1^X$ embeds into $\K_1^Y$ via a computable embedding.  Likewise, $\K_1$ embeds into $\K_1^X$ via a computable embedding for every $X \subseteq \omega$.
\end{proposition}

\begin{proof}
Let $X \leqT Y$.  We show that $\K_1^X$ embeds into $\K_1^Y$ via a computable embedding.  Let $S^1_1(e,y)$ be the injective computable function given by the relativized S-m-n theorem (see~\cite{SoareBookRE}*{Theorem~III.1.5}), where for every $e, y, z \in \omega$ and every $A \subseteq \omega$
\begin{align*}
\Phi_{S^1_1(e, y)}^A(z) \simeq \Phi_e^A(y, z).
\end{align*}

We define an injective computable mapping  $F \colon \omega \imp \omega$ so that for all $a, b, c$:
\begin{align*}
(a \cdot_{\K_1^X} b)\da = c \quad&\Longrightarrow\quad (F(a) \cdot_{\K_1^Y} F(b))\da = F(c) \\
\intertext{and}
(a \cdot_{\K_1^X} b)\ua \quad&\Longrightarrow\quad (F(a) \cdot_{\K_1^Y} F(b))\ua.
\end{align*}
Using the recursion theorem and the fact that computations relative to $X$ can be 
uniformly simulated by computations relative to $Y$, 
we can define an index $e$ such that
\begin{align*}
\Phi_e^Y(a, x) \simeq
\begin{cases}
S^1_1(e, \Phi_a^X(b)) & \text{if $\exists b(S^1_1(e, b) = x \andd \Phi_a^X(b)\da)$}\\
\ua & \text{otherwise}.
\end{cases}
\end{align*}
Define $F$ by $F(a) = S^1_1(e, a)$.  Note that $F$ is injective and computable because $S^1_1$ is injective and computable.

For every $a$ and $b$, we have that
\begin{align*}
F(a) \cdot_{\K_1^Y} F(b) &\simeq \Phi_{F(a)}^Y(F(b)) \simeq \Phi_{S^1_1(e, a)}^Y(S^1_1(e, b)) \simeq \Phi_e^Y(a, S^1_1(e,b)) \\ \\
&\simeq
\begin{cases}
S^1_1(e, \Phi_a^X(b)) & \text{if $\Phi_a^X(b)\da$}\\
\ua &\text{otherwise}
\end{cases} \\ \\
&\simeq
\begin{cases}
F(\Phi_a^X(b)) & \text{if $\Phi_a^X(b)\da$}\\
\ua &\text{otherwise.}
\end{cases}
\end{align*}
Therefore, if $(a \cdot_{\K_1^X} b)\da = c$, this means that $\Phi_a^X(b)\da = c$, in which case $(F(a) \cdot_{\K_1^Y} F(b))\da = F(c)$.  Conversely, if $(a \cdot_{\K_1^X} b)\ua$, this means that $\Phi_a^X(b)\ua$, in which case $(F(a) \cdot_{\K_1^Y} F(b))\ua$ as well.  So $F$ is the desired embedding.

Embedding $\K_1$ into $\K_1^X$ via a computable embedding can be done similarly because oracle Turing machines can uniformly simulate plain Turing machines.
\end{proof}

For an element $a$ of a pca $\A$ and an $n \in \omega$, let $a^n = a \cdot a \cdots a$ denote the $n$-fold application of $a$ to itself.  When $n = 0$, we take $a^0$ to be the element $I = SKK$ of $\A$, which represents the identity function.

\begin{lemma}\label{lem-EmbedComputeHelper}
Let $F \colon \K_1^X \imp \K_1^Y$ be an embedding of $\K_1^X$ into $\K_1^Y$ for 
some $X, Y \subseteq \omega$.  Then there is an index $e \in \K_1^X$ such that 
$F$ is determined by $F(e)$ and $F(0)$.  Specifically, there is an 
index $e$ such that $F(n) = F(e)^n \cdot_{\K_1^Y} F(0)$ for every $n$.
\end{lemma}

\begin{proof}
We show that there is an index $e \in \K_1^X$ such that 
$n = e^n \cdot_{\K_1^X} 0$ for every $n$.  
It follows that $F(n) = F(e)^n \cdot_{\K_1^Y} F(0)$ for every $n$ 
because $F$ is an embedding.

Again we make use of the function $S^1_1(e,y)$ from the relativized S-m-n theorem.  By padding, we may assume that $0$ is not in the range of $S^1_1$.  By the recursion theorem, let $d$ be an index such that 
\begin{align*}
\Phi_d^X(n,x) =
\begin{cases}
S^1_1(d,n+1) & \text{if $x>0$} \\
n & \text{if $x=0$}
\end{cases}
\end{align*}
and define $f(n) = S^1_1(d,n)$. Then 
\begin{align} \label{f}
\Phi_{f(n)}^X(x) =
\begin{cases}
f(n+1) & \text{if $x>0$} \\
n & \text{if $x=0$}
\end{cases}
\end{align}
for every $n$ and $x$.
Let $e = f(1)$, and note that $e \neq 0$.  By induction we have that 
$e^n = f(n)$ for every $n\geq 1$.  For $n=1$ this is by the definition of~$e$, 
and then $e^{n+1} = e^n \cdot e = f(n) \cdot e = f(n+1)$ by the
induction hypothesis and \eqref{f}. 
Hence $e^n \cdot 0 = n$ for every $n$.  This is by definition when $n = 0$, and it is because $e^n \cdot 0 = f(n) \cdot 0 = n$ by
\eqref{f} when $n > 0$.
\end{proof}

\begin{theorem} \label{thm:iff}
Let $X, Y \subseteq \omega$.  Then $\K_1^X$ embeds into $\K_1^Y$ if and only if $X \leqT Y$, in which case $\K_1^X$ embeds into $\K_1^Y$ via a computable embedding.
\end{theorem}

\begin{proof}
Let $X, Y \subseteq \omega$.  If $X \leqT Y$, then $\K_1^X$ embeds into $\K_1^Y$ via a computable embedding by Proposition~\ref{prop-EmbedK1Rel}.

Conversely, suppose that $\K_1^X$ embeds into $\K_1^Y$, and let $F$ be an embedding.  
Let $e \in \K_1^X$ be as in Lemma~\ref{lem-EmbedComputeHelper}, 
so that $F(n) = F(e)^n \cdot_{\K_1^Y} F(0)$ for every $n$.  
Write $a = F(e)$ and $b = F(0)$.  
Notice then that $F \leqT Y$ via the computation 
$n \mapsto a^n \cdot_{\K_1^Y} b$.

The characteristic function of $X$ is $X$-computable, 
so there is an index $c \in \K_1^X$ such that for every $n$,
\begin{align*}
c \cdot_{\K_1^X} n =
\begin{cases}
0 & \text{if $n \notin X$} \\
1 & \text{if $n \in X$}.
\end{cases}
\end{align*}
By applying $F$, we obtain an index $d = F(c)$ in $\K_1^Y$ such that for every $n$,
\begin{align*}
d \cdot_{\K_1^Y} F(n) =
\begin{cases}
F(0) & \text{if $n \notin X$} \\
F(1) & \text{if $n \in X$}.
\end{cases}
\end{align*}
It follows that $X \leqT Y$ because $F \leqT Y$.
\end{proof}

We now show that, for every $X \subseteq \omega$, $\K_1^X$ embeds into $\K_2$ in a way that preserves $\sim_\alpha$ for every ordinal $\alpha$.  An embedding $F \colon \A \imp \B$ of a pca $\A$ into a pca $\B$ extends to map from the closed terms over $\A$ to the closed terms over $\B$ by setting $F(st) = F(s)F(t)$ for all closed terms $s$ and $t$ of $\A$.  We define an embedding $F \colon \K_1^X \imp \K_2$ so that for every ordinal $\alpha$ and all closed terms $s$ and $t$ of $\K_1^X$, $s \sim_\alpha t$ in $K_1^X$ if and only if $F(s) \sim_\alpha F(t)$ in $\K_2$.

\begin{theorem}\label{thm-K1K2}
For every $X \subseteq \omega$, there is an $X$-computable embedding $F \colon \K_1^X \imp \K_2$ such that for all ordinals $\alpha$ and all closed terms $s$ and $t$ of $\K_1^X$, $s \sim_\alpha t$ in $K_1^X$ if and only if $F(s) \sim_\alpha F(t)$ in $\K_2$.  The analogous statement with $\K_1$ in place of $\K_1^X$ holds as well.
\end{theorem}

\begin{proof}
We define an $X$-computable sequence $(f_n : n \in \omega)$ of elements of $\K_2$ with $\la n \ra \sqsubseteq f_n$ for each $n$ such that for all $a, b, c \in \K_1^X$:
\begin{align*}
a \cdot_{\K_1^X} b = c \quad&\Longrightarrow\quad \forall g \sqsupseteq \la b \ra \; (f_a \cdot_{\K_2} g = f_c) \\
(a \cdot_{\K_1^X} b)\ua \quad&\Longrightarrow\quad \forall g \sqsupseteq \la b \ra \; (f_a \cdot_{\K_2} g)\ua.
\end{align*}
It follows that the map $F(a) = f_a$ embeds $\K_1^X$ into $\K_2$ because $f_b \sqsupseteq \la b \ra$ for every $b$.

It suffices to think of an element of $\K_2$ as a code for a total map $\omega^{<\omega} \imp \omega^{<\omega}$, where if $\sigma_0 \mapsto \tau_0$, $\sigma_1 \mapsto \tau_1$, and $\sigma_0 \sqsubseteq \sigma_1$, then also $\tau_0 \sqsubseteq \tau_1$.  If $f$ codes such a map, then $f \cdot_{\K_2} g = h$ if for every $n$ there are a $\sigma \sqsubseteq g$ and a $\tau \sqsubseteq h$ with $|\tau| \geq n$ such that $\sigma \mapsto \tau$.  If there is no such $h$ for a given $g$, then $(f \cdot_{\K_2} g) \ua$.  In the same way, finite strings are thought of as coding finite partial maps on strings.  The only necessary detail from the encoding of $\K_2$ from~\cite{LongleyNormann}*{Section~3.2.1} we need is that if we assume $0$ codes the empty string $\epsilon$, then the partial map coded by $\la n \ra$ is empty for every $n$. 

We define an $X$-computable sequence of finite strings $(f_{n,s} : n, s \in \omega)$ so that for each $n$, $\la n \ra = f_{n, 0} \sqsubseteq f_{n, 1} \sqsubseteq f_{n, 2} \sqsubseteq \cdots$.  In the end, $f_n = \bigcup_s f_{n, s}$ is the desired $f_n$ for each $n$.  Start with $f_{n, 0} = \la n \ra$ for each $n$.  At stage $s+1$, consider each $a \leq s$ and each string $\tau$ with code $\leq s$ that is not yet in the domain of the partial map coded by $f_{a, s}$.  If there are $b, c \leq s$ with $\tau \sqsupseteq \la b \ra$ and $\Phi_{a,s}^X(b) \da = c$, then extend $f_{a, s}$ so as to map $\tau$ to $f_{c,s}$.  Otherwise, extend $f_{a, s}$ so as to map $\tau$ to $\epsilon$.  Let $f_{a, s+1}$ be the result of these extensions.  Also, let $f_{n, s+1} = f_{n, s}$ for all $n > s$.

Suppose that $a \cdot_{\K_1^X} b = c$, and let $s \geq a, b, c$ be large enough so that $\Phi_{a, s}^X(b) = c$.  Then at every stage $t +1 > s$, $f_{a, t}$ is extended to map $\tau$ to $f_{c, t}$ whenever $\tau \sqsupseteq \la b \ra$, $\tau \leq t$, and $\tau$ is not already in the domain of the partial map coded by $f_{a,t}$.  Therefore $f_a \cdot_{\K_2} g = f_c$ for all $g \sqsupseteq \la b \ra$.

Conversely, suppose that $(a \cdot_{\K_1^X} b)\ua$.  Then for every $s$, $f_{a,s}$ maps $\tau$ to $\epsilon$ whenever $\tau \sqsupseteq \la b \ra$ is in the domain of the partial function coded by $f_{a,s}$.  Thus $(f_a \cdot_{\K_2} g)\ua$ for all $g \sqsupseteq \la b \ra$.

We now show that the embedding $F$ preserves $\sim_\alpha$.  First, we have that for every closed term $t$ over $\K_1$ and every $b, c \in \K_1^X$:
\begin{align*}
t \cdot_{\K_1^X} b = c \quad&\Longrightarrow\quad \forall g \sqsupseteq \la b \ra \; (F(t) \cdot_{\K_2} g = F(c)) \\
(t \cdot_{\K_1^X} b)\ua \quad&\Longrightarrow\quad \forall g \sqsupseteq \la b \ra \; (F(t) \cdot_{\K_2} g)\ua.
\end{align*}
This is because if $t\da = a$ then $F(t)\da = F(a)$, so the implications for $t$ follow from the corresponding implications for $a$.  On the other hand, if $t \ua$, then $F(t) \ua$ as well, from which the implications follow.  Therefore, if $t$ is a closed term over $\K_1^X$ and $g \in \K_2$, then $F(t) \cdot_{\K_2} g \sim_0 F(t) \cdot_{\K_2} F(b)$, where $b = g(0)$.

Now proceed by transfinite induction on $\alpha$.  For the base case, we have that $s \sim_0 t$ in $\K_1^X$ if and only if $F(s) \sim_0 F(t)$ in $\K_2$ for all closed terms $s$ and $t$ of $\K_1^X$ because $F$ is an embedding.

For the successor case, assume that $s \sim_\alpha t$ in $\K_1^X$ if and only if $F(s) \sim_\alpha F(t)$ in $\K_2$ for all closed terms $s$ and $t$ of $\K_1^X$.  Now consider two closed terms $s$ and $t$ of $\K_1^X$.  Suppose that $s \sim_{\alpha + 1} t$ in $\K_1^X$.  We want to show that $F(s) \sim_{\alpha + 1} F(t)$ in $\K_2$.  Let $g \in \K_2$, and let $b = g(0)$.  Then $s b \sim_\alpha t b$ because $s \sim_{\alpha + 1} t$, so therefore $F(s b) \sim_\alpha F(t b)$.  Thus,
\begin{align*}
F(s) \cdot_{\K_2} g \sim_0 F(s) \cdot_{\K_2} F(b) \sim_0 F(s b) \sim_\alpha F(t b) \sim_0 F(t) \cdot_{\K_2} F(b) \sim_0 F(t) \cdot_{\K_2} g.
\end{align*}
Thus $F(s) \cdot_{\K_2} g \sim_\alpha F(t) \cdot_{\K_2} g$ for every $g \in \K_2$, so $F(s) \sim_{\alpha + 1} F(t)$.  Conversely, suppose that $s \not\sim_{\alpha + 1} t$ in $\K_1^X$.  Then there is a $b \in \K_1^X$ such that $s b \not\sim_\alpha t b$.  Then $F(s b) \not\sim_\alpha F(t b)$ by the induction hypothesis, so
\begin{align*}
F(s) \cdot_{\K_2} F(b) \sim_0 F(s b) \not\sim_\alpha F(t b) \sim_0 F(t) \cdot_{\K_2} F(b).
\end{align*}
Thus $F(s) \not\sim_{\alpha + 1} F(t)$.

The limit case follows easily from the inductive hypothesis, which completes the proof.

The proof for $\K_1$ in place of $\K_1^X$ is the same, with plain Turing machines in place of oracle Turing machines.
\end{proof}

\begin{definition}
For a pca $\A$ and closed terms $s$ and $t$ over $\A$, let
\begin{align*}
\ord_\A(s, t) =
\begin{cases}
\text{the least ordinal $\alpha$ such that $s \sim_\alpha t$} & \text{if $s \approx t$} \\
\infty & \text{if $s \napprox t$}.
\end{cases}
\end{align*}
\end{definition}
Here $\infty$ is a new symbol, and we define $\alpha < \infty$ for every ordinal $\alpha$.  Notice that 
\begin{align*}
\ord(\A) = \sup\{\ord_\A(s, t) : \text{$s$ and $t$ are closed terms over $\A$ with $\ord_\A(s, t) < \infty$}\},
\end{align*}
by the discussion following Definition~\ref{def-OrdA}.  Furthermore, if $\A$ is not extensional, then
\begin{align*}
\ord(\A) = \sup\{\ord_\A(a, b) : a, b \in \A \andd \ord_\A(a, b) < \infty\}.
\end{align*}

If $\A$ is a sub-pca of $\B$, then $\ord(\A) \leq \ord(\B)$, provided that $\ord_\B(s, t) < \infty$ whenever $s$ and $t$ are closed terms over $\A$ with $\ord_\A(s, t) < \infty$.

\begin{proposition}\label{prop-OrdSubPCA}
Let $\A$ and $\B$ be pca's with $\A$ a sub-pca of $\B$.
\begin{enumerate}[(i)]
\item\label{it-OrdBoundTerm} For all closed terms $s$ and $t$ over $\A$, $\ord_\A(s,t) \leq \ord_\B(s,t)$.

\item\label{it-OrdBoundPCA} If $\ord_\A(s,t) < \infty \Longrightarrow \ord_\B(s,t) < \infty$ for all closed terms $s$ and $t$ over $\A$, then $\ord(\A) \leq \ord(\B)$.

\item\label{it-OrdBoundPCANonExt} If $\A$ and $\B$ are non-extensional and $\ord_\A(a,b) < \infty \Longrightarrow \ord_\B(a,b) < \infty$ for all $a, b \in \A$, then $\ord(\A) \leq \ord(\B)$.
\end{enumerate}
\end{proposition}

\begin{proof}
For~\ref{it-OrdBoundTerm}, a straightforward transfinite induction on $\alpha$ shows that for all ordinals $\alpha$ and all closed terms $s$ and $t$ over $\A$, if $s \sim_\alpha t$ in $\B$, then $s \sim_\alpha t$ in $\A$.  Thus for any fixed closed terms $s$ and $t$ over $\A$, either $\ord_\B(s,t) = \infty$, in which case~\ref{it-OrdBoundTerm} trivially holds; or $\ord_\B(s,t) < \infty$, in which case $s \sim_{\ord_\B(s,t)} t$ in both $\A$ and $\B$, so $\ord_\A(s,t) \leq \ord_\B(s, t)$.

For~\ref{it-OrdBoundPCA}, assume that $\ord_\A(s,t) < \infty \Longrightarrow \ord_\B(s,t) < \infty$ for all closed terms $s$ and $t$ over $\A$.  Then by~\ref{it-OrdBoundTerm}:
\begin{align*}
\ord(\A) &= \sup\{\ord_\A(s, t) : \text{$s$ and $t$ are closed terms over $\A$ with $\ord_\A(s, t) < \infty$}\} \\
&\leq \sup\{\ord_\B(s, t) : \text{$s$ and $t$ are closed terms over $\A$ with $\ord_\B(s, t) < \infty$}\} \\
&\leq \sup\{\ord_\B(s, t) : \text{$s$ and $t$ are closed terms over $\B$ with $\ord_\B(s, t) < \infty$}\} \\
&=\ord(\B).
\end{align*}

For \ref{it-OrdBoundPCANonExt}, if $\A$ and $\B$ are non-extensional, then we can give the same argument as in~\ref{it-OrdBoundPCA} but with elements of the algebras in place of more general closed terms.
\end{proof}

\section{Ordinal analysis of $\K_2$} \label{sec:K2}

In this section we use the results of Section~\ref{sec:embeddings} to 
prove that the closure ordinal of Kleene's second model $\K_2$ is 
equal to~$\omega_1$.

\begin{theorem} \label{thm-K2ord}
$\ord(\K_2) = \omega_1$.
\end{theorem}

\begin{proof}
The proof that $\ord(\K_2) \leq \omega_1$ follows the proof of Theorem~\ref{thm:upper}.  Let $\Omega$ denote the set of 
closed terms over $\K_2$, coded as elements of $\omega^\omega$ in some effective way.  Notice that $s \simeq t$ is an arithmetical property of (the codes for) $s$ and $t$.

Consider the operator 
$\Gamma \colon \PP(\Omega\times\Omega) \imp \PP(\Omega\times\Omega)$
defined by
\begin{align*}
\Gamma(X) = \bigset{(s, t) \in \Omega\times\Omega : \forall x \in \K_2 \; (s x, t x) \in X}.
\end{align*}
Define $\Gamma_\alpha$ for every ordinal $\alpha$ by 
\begin{align*}
\Gamma_0 &= \bigset{(s,t) : s \simeq t} \\
\Gamma_{\alpha+1} &= \Gamma(\Gamma_\alpha) \\ 
\Gamma_\gamma &= \bigcup_{\alpha < \gamma} \Gamma_\alpha \hspace{1cm} \text{for $\gamma$ a limit ordinal.}
\end{align*}
Then for closed terms $s$ and $t$, $s \sim_\alpha t$ if and only if $(s, t) \in \Gamma_\alpha$.

We observe that the operator $\Gamma$ is $\Pi^1_1$ and monotone.  Since $X$ is now a subset of $\omega^\omega$, this has to be understood in terms of Kleene's recursion in higher types.  The predicate $(s x, t x) \in X$ is indeed recursive in $X$, as it requires only one oracle query to $X$.  The closure ordinal for $\Pi^1_1$ monotone operators such as $\Gamma$ is at most $\omega_1$ by Cenzer~\cite{Cenzer}*{Theorem~4.1} (see also \cite{Sacks}*{Corollary~III.8.10}), so $\ord(\K_2) \leq \omega_1$.

One may show that $\omega_1 \leq \ord(\K_2)$ by constructing for each $\alpha < \omega_1$ elements $f_\alpha$ and $g_\alpha$ of $\K_2$ such that $f_\alpha \not\sim_\alpha g_\alpha$ and $f_\alpha \sim_{\alpha +1} g_\alpha$ in a similar spirit to the proof of Theorem~\ref{thm:lower}.  At successor stages, take $f_{\alpha + 1} = K f_\alpha$ and $g_{\alpha + 1} = K g_\alpha$ as before.  At limit stages, choose a sequence of ordinals $\beta_0 < \beta_1 < \beta_2 < \cdots$ with $\alpha = \lim_n \beta_n$, and directly define $f_\alpha$ and $g_\alpha$ so that $f_\alpha \cdot h = f_{\beta_{h(0)}}$ and $g_\alpha \cdot h = g_{\beta_{h(0)}}$ for all $h \in \K_2$.  Instead, we use the embeddings of $\K_1^X$ into $\K_2$ for $X \subseteq \omega$ to show that $\omega_1 \leq \ord(\K_2)$.  Thus consider any $X \subseteq \omega$ and the embedding $F \colon \K_1^X \imp \K_2$ of Theorem~\ref{thm-K1K2}.  Identify $\K_1^X$ with its image under $F$ in order to view $\K_1^X$ as a sub-pca of $\K_2$.  Then Theorem~\ref{thm-K1K2} tells us that for all ordinals $\alpha$ and all closed terms $s$ and $t$ over $\K_1^X$, $s \sim_\alpha t$ in $\K_1^X$ if and only if $s \sim_\alpha t$ in $\K_2$.  In particular, $\ord_{\K_1^X}(s,t) < \infty \Longrightarrow \ord_{\K_2}(s,t) < \infty$ for all closed terms $s$ and $t$ over $\K_1^X$.  Therefore $\ord(\K_1^X) \leq \ord(\K_2)$ by Proposition~\ref{prop-OrdSubPCA}.  Furthermore, $\ord(\K_1^X) = \omega_1^X$ by Theorem~\ref{thm:K1} relativized to $X$.  Thus $\omega_1^X \leq \ord(\K_2)$ for every $X$.  
Now the supremum of the $\omega_1^X$ is $\omega_1$, 
because every $\alpha<\omega_1$ is $X$-computable for some $X$, 
hence $\omega_1 \leq \ord(\K_2)$.
\end{proof}

Consider any $f$ and $g$ in $\K_2$, and let $\Gamma$ be as in the proof of Theorem~\ref{thm-K2ord}.  By Cenzer~\cite{Cenzer}*{Proposition~4.9} (see also \cite{Sacks}*{Theorem~III.8.9}), if $(f, g) \in \Gamma_\alpha$ for some $\alpha$, then $(f, g) \in \Gamma_\alpha$ for an $\alpha < \omega_1^{f \oplus g}$.  That is, if $f \approx g$ in $\K_2$, then $f \sim_{\omega_1^{f \oplus g}} g$.  It would be interesting to give a fine analysis of the $\sim_\alpha$ relations for $\alpha < \omega_1$ on $\K_2$, but we do not pursue this direction here.

We conclude this section with a remark about the effective 
version of $\K_2$. This version is called $\Keff$ in 
Longley and Norman~\cite{LongleyNormann}, and it is obtained 
by restricting $\K_2$ to total computable elements 
$f\in\omega^\omega$. This is again a pca because the combinators 
$K$ and $S$ in $\K_2$ are computable. 

\begin{theorem} \label{ordKeff}
$\ord(\Keff) = \OCK$.
\end{theorem}
\begin{proof}
Using $\emptyset''$, we can index the total computable functions and partially compute the partial application operation of $\Keff$.  
Therefore $\ord(\Keff) \leq \omega_1^{\emptyset''} = \OCK$, 
where the inequality is by Theorem~\ref{thm:countable} and the equality is 
by~\cite{Sacks}*{Corollary~II.7.4}.  The reverse inequality $\OCK \leq \ord(\Keff)$ can be proved in the same way as Theorem~\ref{thm-K2ord}.  The embedding $F \colon \K_1 \imp \K_2$ of Theorem~\ref{thm-K1K2} produces computable elements of $\K_2$.  Thus $F$ embeds $\K_1$ into $\Keff$.  By inspecting the proof of Theorem~\ref{thm-K1K2}, we see that again for all ordinals $\alpha$ and all closed terms $s$ and $t$ of $\K_1$, $s \sim_\alpha t$ in $\K_1$ if and only if $F(s) \sim_\alpha F(t)$ in $\Keff$.  The proof now proceeds as in that of Theorem~\ref{thm-K2ord}, and we conclude that $\OCK = \ord(\K_1) \leq \ord(\Keff)$.
\end{proof}

\section*{Acknowledgments}
The authors thank Emanuele Frittaion, Russell Miller, and Michael Rathjen for helpful discussions.  
This project was partially supported by a grant from the John Templeton Foundation (\emph{A new dawn of intuitionism: mathematical and philosophical advances} ID 60842).  The opinions expressed in this work are those of the authors and do not necessarily reflect the views of the John Templeton Foundation.

\bibliographystyle{amsplain}
\bibliography{ShaferTerwijnPCAOrd}

\vfill

\end{document}